\documentclass[a4paper]{amsart}
\usepackage[all]{xy}
\usepackage{hyperref}
\usepackage{enumerate}
\usepackage{amssymb}
\newtheorem{theorem}{Theorem}[section]
\newtheorem{thmz}{Theorem}             % numbered in introduction
\usepackage[dvips]{graphicx}
 \DeclareMathAlphabet{\mathpzc}{OT1}{pzc}{m}{it}

\newtheorem{proposition}[theorem]{Proposition}
\newtheorem{lemma}[theorem]{Lemma}
\newtheorem{corollary}[theorem]{Corollary}
\theoremstyle{definition}
\newtheorem{definition}[theorem]{Definition}
\newtheorem{example}[theorem]{Example}

\theoremstyle{remark}
\newtheorem{remark}[theorem]{Remark}

\newcommand{\rmap}{\longrightarrow}

 \usepackage[latin1]{inputenc}
 \usepackage{mathrsfs}
 \usepackage{dsfont}
 \usepackage{amsmath}
 \usepackage{amsthm}
 \usepackage{amsfonts}
 \usepackage{amssymb}
 \usepackage[dvips]{graphicx}
 \usepackage{wrapfig}
 \usepackage{layout}
 \usepackage{verbatim}
 \usepackage{alltt}
 \usepackage{xypic}
 \input xy
 \xyoption{all}

\newcommand{\ud}{\mathrm{d}}

\newcommand{\BG}{\mathrm{B}G}
\newcommand{\EG}{\mathrm{E}G}
\newcommand{\nerve}{\mathrm{N}(G)}
\newcommand{\V}{V}
\newcommand{\W}{\mathrm{W}}

\begin{document}
\vspace{15cm}
 \title{The Weil algebra and the Van Est isomorphism}
\author{Camilo Arias Abad} 
\address{Institut f\"ur Mathematik, Universit\"at Z\"urich, Switzerland.}
\email{camio.arias.abad@math.uzh.ch}

\author{Marius Crainic}
\address{Department of Mathematics, Utrecht University, The Netherlands.}
\email{m.crainic@uu.nl}

\thanks{Camilo Arias Abad was partially supported by NWO Grant ``Symmetries and Deformations in Geometry'' and by SNF Grant 200020-121640/1. Marius Crainic was supported by NWO Vidi Grant no. 639.032.712.}

 \begin{abstract} This paper belongs to a series of papers devoted to the study of the cohomology of classifying spaces.
Generalizing the Weil algebra of a Lie algebra and Kalkman's BRST model,
here we introduce the Weil algebra $W(A)$ associated to any Lie algebroid $A$. We then show that
this Weil algebra is related to the Bott-Shulman-Stasheff complex (computing the cohomology of the classifying space)
via a Van Est map and we prove a Van Est isomorphism theorem. As application, we generalize and
find a simpler  more conceptual proof of  the main result of \cite{BUR} on the reconstructions of multiplicative
forms and of a result of \cite{WeXu, Cra2} on the reconstruction of connection 1-forms. This reveals the relevance of the Weil algebra and Van Est maps to the integration and the pre-quantization of Poisson (and Dirac)
manifolds.
\end{abstract}
 \maketitle
\setcounter{tocdepth}{1}
\tableofcontents

\section*{Introduction}

This paper belongs to a series devoted to the study of the cohomology of classifying spaces.
Here we extend the construction of the Weil algebra of a Lie algebra to
the setting of Lie algebroids and we show that one of the standard complexes computing
the cohomology of classifying spaces (the ``Bott-Shulman-Stasheff complex'') is related to the
Weil algebra via a Van Est map. This Van Est map is new even in the case of Lie
groups. As application, we generalize and we find a simpler (and conceptual) proof of
the main result of \cite{BUR} on the integration of Poisson and related structures.

We will be working in the framework of groupoids \cite{McK}. Lie groupoids and Lie algebroids
are a generalization of Lie groups and Lie algebras, generalization which arises in various geometric contexts such as
foliation theory, Poisson geometry, equivariant geometry. Probably the best known example of a Lie groupoid is
the homotopy groupoid of a manifold $M$, which consists of homotopy classes of paths in $M$, each such
path being viewed as an ``arrow'' from the initial point to the ending point. In the presence of a foliation on $M$,
by restricting to paths inside leaves and leafwise homotopy (or holonomy), one arrives at the groupoids which
are central to foliation theory \cite{Haef}. When $G$ is a Lie group of symmetries of a manifold
$M$, it is important to retain not only the group structure of $G$ but also at the way that the points of $M$ are affected
by the action. More precisely, one considers pairs $(g, x)\in G\times M$ and one views such a pair as an ``arrow'' from
$x$ to $gx$; one obtains a groupoid $G\ltimes M$, called the action groupoid. In Poisson geometry, the relevance of Lie groupoids \cite{Wein2} is a bit more subtle- they
arise through their infinitesimal counterpart (Lie algebroids) after an integration process. They also arise as the canonical solution to the problem of finding
symplectic realizations of Poisson manifolds. Indeed, an important feature of the resulting groupoids
is that they carry a canonical symplectic structure (for instance, if $M$ is endowed with the zero-Poisson structure, the resulting
groupoid is $T^*M$ with the canonical symplectic structure). Conceptually, such symplectic forms can be seen as arising after integrating
certain infinitesimal data, called IM forms (see Example \ref{IM}). This integration step is based on the main result of \cite{BUR} mentioned above, result
that will be generalized and proved more conceptually in this paper. The key remark is that the rather mysterious equations that IM forms have to
satisfy are nothing but cocycle equations in the Weil algebra of the associated Lie algebroid.

The classifying space $BG$ \cite{Sea} of a Lie groupoid $G$ has the same defining property as in the case of Lie groups: it
is the base of an universal principal $G$-bundle $EG\rmap BG$. It is unique up to homotopy, and there
are several known constructions of $BG$ as a topological space and of $EG$ as topological $G$-bundle
- e.g. via a Milnor-type construction \cite{Haef} or via simplicial methods \cite{Bott, Bott2}. Due to the universality of $EG$,
the cohomology of $BG$  is the algebra of universal characteristic classes for $G$-bundles. For instance, if $G$ is a compact Lie group,
then $H^*(BG)= S(\mathfrak{g}^*)^G$ is the space of $G$-invariant polynomials on the Lie algebra $\mathfrak{g}$ of $G$,
which is the source of the Chern-Weil construction of characteristic classes via connections. The cohomology of classifying
spaces is interesting also from the point of view of equivariant cohomology; for instance, for the Lie groupoid
$G\ltimes M$ associated to an action of a Lie group $G$ on a manifold $M$, the cohomology of $B(G\ltimes M)$ coincides
with the equivariant cohomology of $M$. For general Lie groupoids $G$, simplicial techniques \cite{Sea} provides us
with a huge but explicit complex computing $H^*(BG)$, known as the Bott-Shulman-Stasheff complex \cite{Bott2}, which is a double complex suggestively denoted
\[ \Omega^{\bullet}(G_{\bullet}).\]

We can now explain the words in the title and our main results.
First of all, the classical Weil algebra $W(\mathfrak{g})$ associated to a Lie group $G$ (or, better, to its Lie algebra $\mathfrak{g}$)
arises as the algebraic model for the ``DeRham complex'' of the total space $EG$ \cite{Cartan} (see also \cite{GS, Getzler}). From the point of view of
characteristic classes, the role of $W(\mathfrak{g})$ is to provide an explicit and geometric construction of such classes (the Chern-Weil construction).
From the point of view of equivariant cohomology, it is useful in constructing explicit geometric models for equivariant cohomology
(such as the Cartan model). As we have already mentioned, the first aim of this paper is to extend the construction of the Weil algebra from Lie algebras to
Lie algebroid- for any Lie algebroid $A$, its Weil algebra, denoted
\[ W^{\bullet, \bullet}(A),\]
will be a differential bi-graded algebra.

Next, the classical Van Est map \cite{Est} relates the differentiable cohomology of a Lie group $G$
to its infinitesimal counterpart, i.e. to the Lie algebra cohomology of the Lie algebra $\mathfrak{g}$ of $G$.
It is an isomorphism up to degree $k$ provided $G$ is $k$-connected. The Van Est isomorphism
extends to Lie groupoids and Lie algebroids \cite{WeXu, Cra1} without much trouble.
What is interesting to point out here is that the complex computing the differentiable cohomology of $G$ is just the
first line $\Omega^{0}(G_{\bullet})$ of the Bott-Shulman-Stasheff complex, while the complex computing the Lie algebroid cohomology
is just the first line $W^{\bullet, 0}(A)$ of the Weil algebra.
With these in mind, the second aim of the paper is to extend the classical Van Est map to a map of bi-graded
differential algebras
\[ V: \Omega^{\bullet}(G_{\bullet})\rmap W^{\bullet, \bullet}(A),\]
for any Lie groupoid $G$ with associated Lie algebroid $A$. Topologically, this map is just an explicit model
for the map induced by the pull-back along the projection $\pi: EG\rmap BG$.
However, what is more interesting is that the Van Est isomorphism holds not only along the first line, but
along all lines. More precisely, we will prove the following:

\begin{thmz}%{\bf\ref{Van Est}}
Let $G$ be a Lie groupoid with Lie algebroid $A$ and $k$-connected
source fibers. When restricted to any $q$-line ($q$ arbitrary), the Van
Est map induces an isomorphism in cohomology
\begin{equation*}
V:H^p(\hat{\Omega}^q(G_{\bullet}))\rightarrow H^p(W^{\bullet,q}(A)),
\end{equation*}
for all $p \leq k$, while for $p= k+1$ this map is injective.
\end{thmz}

As a consequence, we prove the following generalization of the
reconstruction result for multiplicative $2$-forms which appears in
\cite{BUR}.

\begin{thmz}\label{xxx2}
Let $G$ be a source simply connected Lie groupoid over $M$ with Lie
algebroid $A$ and let $\phi \in \Omega^{k+1}(M)$ be a closed form.
Then there is a one to one correspondence between:
\begin{itemize}
\item multiplicative forms $\omega \in \Omega^k(G)$ which are $\phi$-relatively closed.
\item $C^{\infty}(M)$-linear maps $\tau:\Gamma(A)\rightarrow \Omega^{k-1}(M)$
satisfying the equations:
\begin{eqnarray}
& i_{\rho(\beta)}(\tau(\alpha))  =  -i_{\rho(\alpha)}(\tau(\beta)),\nonumber\\
& \tau([\alpha,\beta])  =  L_{\alpha}(\tau(\beta))-L_{\beta}(\tau(\alpha))
+d_{DR}(i_{\rho(\beta)}\tau(\alpha))+ i_{\rho(\alpha)\wedge
\rho(\beta)}(\phi).\nonumber
\end{eqnarray}
for all $\alpha, \beta\in \Gamma(A)$.
\end{itemize}
\end{thmz}

This theorem reveals the relationship
between the Van Est map $V$ and the integrability of Poisson and
Dirac structures. This relationship can be summarized as follows:
the Lie algebroid associated to a Poisson (or Dirac) structure comes together with
a tautological cocycle living in the Weil algebra of the associated
Lie algebroid. Integrating the Lie algebroid to a Lie groupoid $G$, the Van Est isomorphism
tells us that the tautological cocycle integrates to a cocycle on the
Bott-Shulman-Stasheff complex of the groupoid- which, in this case, is a multiplicative two-form
on the groupoid, making $G$ a symplectic ( or presymplectic \cite{BUR}) groupoid.

Another application of our Van Est isomorphism is a generalization (and another proof) of
the result of \cite{Cra2} on the construction of connection 1-forms on prequantizations.
Conceptually, it answers the following question: given $\omega\in \Omega^k(G)$ multiplicative and
closed, when can one write $\omega= \ud\theta$ with $\theta\in
\Omega^{k-1}(G)$ multiplicative? We prove the following result which, when $k= 2$, coincides with the
result of \cite{Cra2} that we have mentioned.

\begin{thmz}\label{xxx3}Let $G$ be a source simply connected Lie groupoid over $M$ with Lie
algebroid $A$ and let $\omega\in \Omega^k(G)$ be a closed
multiplicative $k$-form. Then there is a 1-1 correspondence between:
\begin{itemize}
\item $\theta\in \Omega^{k-1}(G)$ multiplicative satisfying $\ud(\theta)= \omega$.
\item $C^{\infty}(M)$-linear maps $l: \Gamma(A)\rmap \Omega^{k-2}(M)$ satisfying
\begin{eqnarray}
&i_{\rho(\beta)}(l(\alpha))=-i_{\rho(\alpha)}(l(\beta)),\label{mmk-1}\\
& c_{\omega}(\alpha, \beta)= - l([\alpha, \beta])+
L_{\rho(\alpha)}(l(\beta))- L_{\rho(\beta)}(l(\alpha))+
d_{DR}(i_{\rho(\beta)}l(\alpha)). \label{mmk-2}
\end{eqnarray}
\end{itemize}
where $c_{\omega}(\alpha, \beta)= i_{\rho(\alpha)\wedge \rho(\beta)}(\omega)|_{M}$. The correspondence is given by
\begin{equation*}
l(\alpha)= - i_{\alpha}(\theta)|_{M}.
\end{equation*}
\end{thmz}

%\vspace*{.1in}

We should now point out what this paper does not achieve. In the case of a compact Lie group $G$, 
the cohomology of $BG$ is isomorphic to $W(\mathfrak{g})_{\textrm{bas}}\cong S(\mathfrak{g}^*)$ where ``bas'' refers to the basic sub-complex, which consists of elements
that are horizontal and invariant with respect to $G$. More generally, 
for an action groupoid $G\ltimes M$, the cohomology of the classifying space is the equivariant cohomology $H^{*}_{G}(M)$, and the Weil algebra associated to the Lie algebroid $\mathfrak{g}\ltimes M$ of $G\ltimes M$ is $W(\mathfrak{g}\ltimes M)= W(\mathfrak{g})\otimes \Omega(M)$.
In this situation one can still define the basic subcomplex $W(\mathfrak{g}\ltimes M)_{\textrm{bas}}$, and it is isomorphic to Cartan's equivariant De-Rham complex
of $M$. Hence $W(\mathfrak{g}\ltimes M)_{\textrm{bas}}$ computes the equivariant cohomology of the action, provided $G$ is compact. We discuss this construction in Subsection \ref{subseq: Equivariant cohomology}, more details can be found in \cite{GS, KT}. Back to a general Lie groupoid $G$ with Lie algebroid $A$, it is natural to expect a construction of a ``basic subcomplex'' of $W(A)$ which
computes the cohomology of $BG$, at least under some compactness assumptions. However, that does not seem to work in an obvious way. 
%The problem is that,
%for a general algebroid $A$, there is no canonical basic sub-complex of $W(A)$. One  could try to use a connection (satisfying extra-flatness conditions) in order to
%define such a subcomplex, but we believe that approach is too naive, besides being un-natural. What we expect is that, in this generality, one should use the full structure
%of the Weil algebra- a structure that seems to be more subtle then in the case of group actions and which still has to be discovered.
On the other hand, we would like to mention here that in \cite{Ari-Cra2}, we did find a generalization of Bott's spectral sequence using the notion of ``representations up to homotopy'' to define the coadjoint representation of a Lie groupoid. It seems possible that these two approaches can be combined to obtain a general Cartan type model
for the cohomology of classifying spaces of Lie groupoids, but at the moment we do not know how to do it.
%In order to find a Cartan-type model for the cohomology of classifying spaces of LIe groupoids it seems one needs to combine the two approaches. In particular, we believe that the full structure of $W(A)$ that we referred to above is a ``up to homotopy structure''. Then one should look at the complex consisting of ``differentiable cocycles on $G$ with values in $W(A)$'' (with the ``basic complex'' consisting of cocycles on $G$ with values in the
%symmetric powers of the adjoint representation). This would also agree with what happens in the case of an action groupoid $G\ltimes M$: while the Cartan model 
%$\Omega_{G}(M)$ only works when $G$ is compact, Getzler has shown \cite{Getzler2} that the correct generalization in the non-compact case requires the use of differentiable group-cocycles with values in $S(\mathfrak{g}^*)\otimes \Omega(M)$. On top, the spectral sequence that Getzler derives coincides in this case with our Bott spectral sequence \cite{Ari-Cra2}. So, although the main pieces for finding a Cartan (or better: Cartan-Getzler) model for the cohomology of classifying spaces seem to fit well together, we do not know how to assemble them yet. 

\vspace*{.1in}

This paper is organized as follows. In the first two sections we recall
some standard facts about classifying spaces of Lie groups, the classical Weil algebra, equivariant
cohomology, Lie groupoids and their classifying
spaces, the Bott-Shulman-Stasheff complex. In Section \ref{the Weil algebra} we
introduce the Weil algebra of a Lie algebroid (Definition
\ref{def-W}) by giving an
explicit choice free description. Then we point out the local formulas and
the relationship with the adjoint representation
and the algebra described in \cite{Ari-Cra1} using representations up to homotopy. Section
\ref{section the Van Est map} contains the definition of the Van Est
map $\V:\hat{\Omega}^{\bullet}(G_{\bullet})\rightarrow W(A)$
(Theorem \ref{proposition Van Est}). In Section \ref{section the Van
Est isomorphism} we prove the isomorphism theorem for the
homomorphism induced in cohomology by the Van Est map (Theorem
\ref{Van Est}). In Section \ref{integrability}  we explain the relation between the Van Est map and
the integration of IM forms to multiplicative forms
(Theorems \ref{cohomological integration1} and
\ref{cohomological integration2}). Finally, Section\ref{appendix} is
an appendix where we describe an infinite dimensional version of
Kalkman's BRST algebra which is used throughout the paper.

At this point we would also like to mention that, while this work has been carried out,
various people used the supermanifold language to construct the Weil algebra of Lie
algebroids. We found out about such descriptions from D. Roytenberg and P. Severa
(unpublished); A Van Est map using supermanifolds was constructed also by A. Weinstein (unpublished notes).
This construction appears in the PhD thesis of R. Mehta \cite{Meh}, where he describes the Weil algebra of a Lie algebroid in terms of supergeometry. That algebra is isomorphic to the one we present here.

{\bf Acknowledgements}
We would like to thank Henrique Bursztyn for suggesting to us Corollary \ref{cor:hen}.

\section{Lie groups: reminder on the classical Weil algebra and classifying spaces}\label{preliminaries Van Est}

In this section we recall some standard facts about classifying
spaces of Lie groups, Weil algebras
and equivariant cohomology. As references, we mention here \cite{Bott, Cartan, Bott, Getzler, KT}.

\subsection{The universal principal bundle} Associated to any Lie group
$G$ there is a classifying space $BG$ and an universal principal
$G$-bundle $EG\rmap BG$. These have the following universal
property. For any space $M$ there is a bijective correspondence
\[ [ M; BG]\ \  \stackrel{1-1}{\longleftrightarrow} \ \ \textrm{Bun}_{G}(M)\]
between homotopy classes of maps $f: M\rmap BG$ and isomorphism
classes of principal $G$-bundles over a $M$. This correspondence
sends a function $f$ to the pull-back bundle $f^*(EG)$. The
universal property determines $EG\rmap BG$ uniquely up to homotopy.
Another property that determines $EG$, hence also $BG$, uniquely, is
that $EG$ is a free, contractible $G$-space. There are explicit
combinatorial constructions for the classifying bundle of a group,
we will say more about this in a moment.

The cohomology of $BG$ is the universal algebra of characteristic
classes for principal $G$-bundles. Indeed, from the universal
property, any such bundle $P\rmap M$ is classified by a map $f_P:
M\rmap BG$. Although $f_P$ is unique only up to homotopy, the map
induced in cohomology
\[ f_{P}^{*}: H^{\bullet}(BG)\rmap H^{\bullet}(M)\]
only depends on $P$ and is called the characteristic map of $P$. Any
element $c\in H^{\bullet}(BG)$ will induce the $c$-characteristic
class of $P$, $c(P):= (f_{P})^*(c)\in H^{\bullet}(M)$. This is one
of the reasons one is often interested in more explicit models for
the cohomology of $BG$. When $G$ is compact, a theorem of Borel
asserts that
\[ H^{\bullet}(BG)\cong S(\mathfrak{g^*})^G.\]

Moreover, the map $f_{P}^{*}$ can be described geometrically. This
is the Chern-Weil construction of characteristic classes in
$H^{\bullet}(M)$ out of invariant polynomials on $\mathfrak{g}$,
viewed as a map
\begin{equation*}
S(\mathfrak{g}^*)^G\rmap H(M).
\end{equation*}

\subsection{The Weil algebra} Regarding the cohomology of $BG$ and the
construction of characteristic classes, the full picture is achieved
only after finding a related model for the \textit{De Rham
cohomology} of $EG$. This is the Weil algebra $W(\mathfrak{g})$ of
the Lie algebra of $G$ \cite{Cartan, GS, KT}. As a graded algebra, it is defined as
\[ W^n(\mathfrak{g})=\bigoplus_{2p+q= n} S^p(\mathfrak{g}^*)\otimes \Lambda^q(\mathfrak{g}^*).\]
We interpret its elements as polynomials $P$ on $\mathfrak{g}$ with
values in $\Lambda(\mathfrak{g}^*)$, but keep in mind that the
polynomial degree counts twice. The Weil algebra can also be made
into bi-graded algebra, with
\[ W^{p, q}(\mathfrak{g})= S^q(\mathfrak{g}^*)\otimes \Lambda^{p-q}(\mathfrak{g}^*), \]
and its differential $d_W$ can be decomposed into two components:
\[ d_W= d^{\textrm{h}}_{W}+ d^{\textrm{v}}_{W}.  \]
Here, $d^{\textrm{v}}_{W}$ increases $q$ and is given by
\[ d^{\textrm{v}}_{W}(P)(v)= i_v(P(v)),\]
while $d^{\textrm{h}}_{W}$ increases $p$ and is defined as a Koszul
differential of $\mathfrak{g}$ with coefficients in
$S^q(\mathfrak{g}^*)$- the symmetric powers of the coadjoint
representation. To be more explicit, it is customary to use
coordinates. A basis $e^1, \ldots , e^n$ for $\mathfrak{g}$ gives
structure constants $c_{j k}^{i}$. With this choice,
$W(\mathfrak{g})$ can be described as the free graded commutative
algebra generated by elements $\theta^1, \ldots , \theta^n$ of
degree $1$ and $\mu^1, \ldots , \mu^n$ of degree $2$, with
differential:
\begin{eqnarray*}
d^{\textrm{v}}_{W}(\theta^i) & = & \mu^i,\\
d^{\textrm{v}}_{W}(\mu^i) & = & 0,\\
d^{\textrm{h}}_{W}(\theta^i) &=&-\frac{1}{2}\sum_{j, k}c^i_{jk}\theta^j\theta^k,\\
d^{\textrm{h}}_{W}(\mu^i)&=&-\sum_{j, k}c^i_{jk}\theta^j\mu^k.
\end{eqnarray*}
Of course, $\theta^1, \ldots , \theta^n$ is just the induced basis
of $\Lambda^1(\mathfrak{g}^*)$, while $\mu^1, \ldots , \mu^n$ is the
one of $S^1(\mathfrak{g}^*)$. The cohomology of $W(\frak{g})$ is
$\mathbb{R}$ concentrated in degree zero, as one should expect from
the fact that $EG$ is a contractible space.

\subsection{$\mathfrak{g}$-DG algebras} To understand why
$W(\mathfrak{g})$ is a model for the De Rham complex of $EG$, one
has to look at the structure present in the De Rham algebras of
principal $G$-bundles. This brings us to the notion of
$\mathfrak{g}$-DG algebras (cf. e.g. \cite{GS, KT}). A $\mathfrak{g}$-DG algebra is a
differential graded algebra $(\mathcal{A}, d)$ (for us $\mathcal{A}$
lives in positive degrees and $d$ increases the degree by one),
together with
\begin{itemize}
\item  Degree zero derivations $L_v$ on the DG-algebra $\mathcal{A}$ ,
depending linearly on $v\in \mathfrak{g}$,  which induce an  action
of $\mathfrak{g}$ on $\mathcal{A}$.
\item Degree $-1$ derivations $i_v$ on the DG-algebra $\mathcal{A}$
such that for all $v, w\in \mathfrak{g}$
\[ [i_v, i_w]= 0, \ \ [L_v, i_w]= i_{[v, w]} \]
and such that they determine the Lie derivatives by Cartan's formula
\[di_v+ i_vd= L_v .\]
\end{itemize}
The basic subcomplex of a $\mathfrak{g}$-DG algebra $\mathcal{A}$ is
defined as
\[ \mathcal{A}_{\textrm{bas}}:= \{\omega\in \mathcal{A} : i_v\omega = 0, L_v\omega = 0,\ \ \forall  v\in \mathfrak{g}\}.\]
Of course, the De Rham complexes $\Omega(P)$ of $G$-manifolds $P$
are the basic examples of $\mathfrak{g}$-DG algebras. In this case
$L_v$ and $i_v$ are just the usual Lie derivative and interior
product with respect to the vector field $\rho(v)$ on $P$ induced
from $v$ via the action of $G$. If $P$ is a principal $G$-bundle
over $M$, then $\Omega(P)_{\textrm{bas}}$ is canonically isomorphic
to
 $\Omega(M)$.

The Weil algebra $W(\mathfrak{g})$ is a model for the De Rham
complex of $EG$. Indeed, $W(\mathfrak{g})$ is canonically a
$\mathfrak{g}$-DG algebra. The operators $L_v$ are the unique
derivations which, on $S^1(\mathfrak{g}^*)$ and on
$\Lambda^1(\mathfrak{g}^*)$, are just the coadjoint action. The
operators $i_v$ are just the standard interior products on the
exterior powers
hence they act trivially on $S(\mathfrak{g}^*)$.

\subsection{Equivariant cohomology} \label{subseq: Equivariant cohomology} The Weil algebra is useful because it
provides explicit models that compute equivariant cohomology (cf. e.g. \cite{GS, Getzler}). Given
a $G$-space $M$, the pathological quotient $M/G$ is often replaced
by the homotopy quotient
\[ M_G= (EG\times M)/G .\]
Here, $EG$ should be thought of as a replacement of the one-point
space $pt$ with a free $G$-space which has the same homotopy as
$pt$. The equivariant cohomology of $M$ is defined as
\[ H^{\bullet}_{G}(M):= H^{\bullet}(M_G).\]
When $M$ is a manifold, one would like to have a more geometric De
Rham model computing this cohomology. This brings us at Cartan's
model for equivariant cohomology. One defines the equivariant De
Rham complex of a $G$-manifold $M$ as
\[ \Omega_G(M)= (S(\mathfrak{g})\otimes \Omega(M))^G  ,\]
the space of $G$-invariant polynomials on $\mathfrak{g}$ with values
in $\Omega(M)$. The differential $d_{G}$ on $\Omega_G(M)$ is very
similar to the one of $W(\mathfrak{g})$:
\[ d_G(P)(v)= d_{DR}(P(v))+ i_v(P(v)).\]

Let us recall how the Weil algebra leads naturally to the Cartan model.
The idea is quite simple. $W(\mathfrak{g})$ is a model for the De
Rham complex of $EG$, the similar model for $EG\times M$ is
$W(\mathfrak{g})\otimes \Omega(M)$- viewed as a $\mathfrak{g}$-DG
algebra with operators
\[ i_v= i_v\otimes 1 + 1\otimes i_v,\ \ L_v= L_v\otimes 1+ 1\otimes L_v ,\]
and with differential
\[ d= d_W\otimes 1+ 1\otimes d_{DR} .\]
The resulting basic subcomplex should provide a model for the
cohomology of the homotopy quotient. Indeed, there is an
isomorphism:
\[ (W(\mathfrak{g})\otimes \Omega(M))_{\textrm{bas}}\cong \Omega_G(M).\]
This is best seen using Kalkman's BRST model, which is a
perturbation of $W(\mathfrak{g})\otimes \Omega(M)$. As a
$\mathfrak{g}$-DG algebra, it has
\[ i^{K}_v= i_v\otimes 1, \ L^{K}_v= L_v \otimes 1.\]
To describe its differential, we use a basis for $\mathfrak{g}$ as
above and set:
\[ d^{K}= d+ \theta^a\otimes L_{e^a} - \omega^a \otimes  i_{e^a} .\]
The resulting basic subcomplex  is $\Omega_G(M)$. In fact, there is
an explicit automorphism $\Phi$ of $W(\mathfrak{g})\otimes
\Omega(M)$  (the Mathai-Quillen isomorphism \cite{MQ}) which transforms $i_v,
L_v$ and $d$ into Kalkman's $i^{K}_v, L^{K}_v$, $d^K$.

\section{Lie groupoids: definitions, classifying spaces and the Bott-Shulman-Stasheff complex}
\label{oids}

In this section we recall some basic definitions on Lie groupoids and the construction
of the Bott-Shulman-Stasheff complex. As references, we use \cite{Bott2, McK, Moe}.

\subsection{Lie groupoids}  A groupoid is a category in which all arrows
are isomorphisms. A Lie groupoid is a groupoid in which the space of
objects $G_0$ and the space of arrows $G_1$ are smooth manifolds and
all the structure maps are smooth. More explicitly, a Lie groupoid
is given by a manifold of objects $G_0$ and a manifold of arrows
$G_1$ together with smooth maps $s,t:G_1 \rightarrow G_0$ the source
and target map, a composition map $m:G_1\times_{G_0} G_1\rightarrow
G_1$, an inversion map $\iota:G\rightarrow G$  and an identity map
$\epsilon:G_0 \rightarrow G_1$ that sends an object to the
corresponding identity. These structure maps should satisfy the
usual identities for a category. The source and target maps are
required to be surjective submersions and therefore the domain of
the composition map is a manifold. We will usually denote the space
of objects of a Lie groupoid by $M$ and say that $G$ is a groupoid
over $M$. We say that a groupoid is source $k$-connected if the
fibers of the source map are $k$-connected.

\begin{example}\rm \
A Lie group $G$ can be seen as a Lie groupoid in which the space of
objects is a point. Associated to any manifold $M$ there is the pair
groupoid $M\times M$ over $M$ for which there is exactly one arrow
between each pair of points. If a Lie group $G$ acts on a manifold
$M$ there is an associated action groupoid over $M$ denoted $G
\ltimes M$ whose space of objects is $G\times M$. Other important
examples of groupoids are the holonomy and monodromy groupoids of
foliations, the symplectic groupoids of Poisson geometry- some of
which arise via their infinitesimal counterparts, Lie algebroids.
\end{example}

\subsection{Lie algebroids} A Lie algebroid over a manifold $M$ is a
vector bundle $\pi:A \rightarrow M$ together with a bundle map
$\rho:A \rightarrow TM$, called the anchor map and a Lie bracket in
the space $\Gamma(A)$ of sections of $A$ satisfying Leibniz
identity:
\begin{equation*}
[\alpha,f\beta]=f[\alpha,\beta]+\rho(\alpha)(f)\beta,
\end{equation*}
for every $\alpha,\beta \in \Gamma(A)$ and $f \in C^{\infty}(M)$. It
follows that $\rho$ induces a Lie algebra map at the level of
sections. Examples of Lie algebroids are Lie algebras, tangent
bundles, Poisson manifolds, foliations and Lie algebra actions.
Given a Lie groupoid $G$, its Lie algebroid $A= A(G)$ is defined as
follows. As a vector bundle, it is  the restriction of the kernel of
the differential of the source map to $M$. Hence, its fiber at $x\in
M$ is the tangent space at the identity arrow $1_x$ of the source
fiber $s^{-1}(x)$.  The anchor map is the differential of the target
map. To describe the bracket, we need to discuss invariant vector
fields. A right invariant vector field on a Lie groupoid $G$ is a
vector field $\alpha$ which is tangent to the fibers of $s$ and such
that, if $g,h$ are two composable arrows and we denote by
$R^{\textrm{h}}$ the right multiplication by $h$, then
\begin{equation*}
\alpha (gh)=D_g(R^{\textrm{h}})(\alpha (g)).
\end{equation*}
The space of right invariant vector fields is closed under the Lie
bracket of vector fields
and is isomorphic to $\Gamma(A)$. Thus, we get the desired Lie bracket on $\Gamma(A)$. \\

Unlike the case of Lie algebras, Lie's third theorem does not hold
in general. Not every Lie algebroid can be integrated to a Lie
groupoid. The precise conditions for the integrability are described
in \cite{Cra-Fer}. However, Lie's first and second theorem do hold.
Due to the first one- which says that if a Lie algebroid is
integrable then it admits a canonical source simply connected
integration- one may often assume that the Lie groupoids under
discussion satisfy this simply-connectedness assumption.

\subsection{Actions} A left action of a groupoid $G$ on a space $P
\stackrel{\nu}{\rightarrow} M$ over $M$ is a map $G_1\times_{M} P
\rightarrow P$ defined on the space $G_1\times_{M} P$ of pairs $(g,
p)$ with $s(g)= \nu(p)$, which satisfies $\nu(gp)= t(g)$ and the
usual conditions for actions. Associated to the action of $G$ on
$P\rightarrow M$ there is the action groupoid, denoted $G\ltimes P$.
The base space is $P$, the space of arrows is $G_1\times_{M}P$, the
source map is the second projection and the target map is the
action. The multiplication in this groupoid is $(g, p)(h, q)= (gh,
q)$.

\begin{example}\label{action-nerve}\rm \
For a Lie groupoid $G$, we denote by $G_k$ the space of strings of
$k$ composable arrows of $G$. When we write a string of $k$
composable arrows $(g_1,\dots,g_k)$ we mean that
$t(g_i)=s(g_{i-1})$. Since the source and target maps are
submersions, all the $G_k$ are manifolds. Each of the $G_k$'s
carries a natural left action. First of all, we view $G_k$ over $M$
via the map
\[ t: G_k\rightarrow M, \ \ (g_1, \ldots , g_k)\mapsto t(g_1).\]
The left action of $G$ on $G_k\stackrel{t}{\rightarrow} M$ is just
\[ g (g_1, g_2, \ldots , g_k)= (gg_1, g_2, \ldots , g_k).\]
We denote by $P_{k-1}(G)$ the corresponding action groupoid.
\end{example}

Analogous to actions of Lie groupoids, there is the notion of
infinitesimal actions. An action of a Lie algebroid $A$ on a space
$P \stackrel{\nu}{\rightarrow} M$ over $M$ is a Lie algebra map
$\rho_P: \Gamma(A)\rightarrow \frak{X}(P)$, into the Lie algebra of
vector fields on $P$, which is $C^{\infty}(M)$-linear in the sense
that
\[ \rho_P(f\alpha)= (f\circ \nu) \rho_P(\alpha), \]
for all $\alpha\in \Gamma(A)$, $f\in C^{\infty}(M)$. Note that this
last condition is equivalent to the fact that $\rho_P$ is induced by
a bundle map $\nu^*A\rightarrow TP$.

As in the case of Lie groupoids, associated to an action of $A$ on
$P$ there is an action algebroid $A\ltimes P$ over $P$. As a vector
bundle, it is just the pull-back of $A$ via $\mu$. The anchor is
just the infinitesimal action $\rho_P$. Finally, the bracket is
uniquely determined by the Leibniz identity and
\[ [\mu^*(\alpha), \mu^*(\beta)]= \mu^*([\alpha, \beta]),\]
for all $\alpha, \beta\in \Gamma(A)$.

As expected, an action of a groupoid $G$ on a space $P
\stackrel{\nu}{\rightarrow} M$ over $M$ induces an action of the Lie
algebroid $A$ of $G$ on $P$. As a bundle map $\rho_P:
\nu^*A\rightarrow TP$ it is defined fiberwise as the differential at
the identity of the map
\[ s^{-1}(\nu(p))\rightarrow P ,\ \ g\mapsto gp.\]
Moreover, the Lie algebroid of $G\ltimes P$ is equal to $A\ltimes
P$.

\begin{example}\label{action-gr-inf}\rm \
For the action of $G$ on $G_k$ (Example \ref{action-nerve}), the
resulting algebroid $A\ltimes G_k$ is just the foliation
$\mathcal{F}_k$ of $G_k$ by the fibers of the map $d_0: G_k\rmap
G_{k-1}$, which deletes $g_1$ from $(g_1, \ldots , g_k)$.
\end{example}

\subsection{Classifying spaces}  We now recall the construction of the
classifying space of a Lie groupoid, as the geometric realization of
its nerve \cite{Sea}. First of all, the nerve of $G$, denoted $N(G)$, is the
simplicial manifold whose $k$-th component is $N_k(G)= G_k$, with
the simplicial structure given by the face maps:
\begin{equation*}
d_i(g_1,\dots,g_k)=
\begin{cases}
(g_2,\dots,g_k) & \text{if } i=0,\\
(g_1,\dots,g_ig_{i+1},\dots,g_k) & \text{if } 0<i<k,\\
(g_1,\dots,g_{k-1}) & \text{if } i=k,\\
\end{cases}
\end{equation*}
and the degeneracy maps:
\begin{equation*}
s_i(g_1,\dots,g_k)=(g_1\dots,g_{i},1,g_{i+1},\dots,g_k)
\end{equation*}
for $0\leq i\leq k$.\\

The thick geometric realization of a simplicial manifold
$X_{\bullet}$, is defined as the quotient space
\begin{equation*}
||X_{\bullet}||=(\coprod_{k\geq 0} X_k \times \Delta^k) / \sim,
\end{equation*}
obtained by identifying $(d_i(p),v)\in X_k \times \Delta^k$ with
$(p,\delta_i(v))\in  X_{k+1} \times \Delta^{k+1}$ for any $p\in
X_{k+1}$ and any $v\in \Delta^k$. Here $\Delta^k$ denotes the
standard topological $k$-simplex and $\delta_i: \Delta^k\rmap
\Delta^{k+1}$ is the inclusion as the $i$-th face. The classifying
space of a Lie groupoid $G$ is defined as
\begin{equation*}
\BG= || N(G) ||.
\end{equation*}

\begin{definition}
The universal $G$-bundle of a Lie groupoid $G$ is defined as
\begin{equation*}
EG=B(\mathrm{P_0}(G)),
\end{equation*}
the classifying space of the groupoid associated to the action of $G$ on itself.
\end{definition}

The nerve of $P_0(G)$ satisfies $(\mathrm{P_0}(G))_k=G_{k+1}$ which, for
each $k$, is a (principal) $G$-space over $G_k$. Moreover, each face
map is $G$-equivariant with respect to the right action, see Example
\ref{action-nerve}. It follows that $EG$ is a principal $G$-bundle
over $BG$
$$\xymatrix{
   G_1 \ar@ <-2 pt>[d] \ar@ <+2 pt>[d]&\EG \ar[d]_{\pi}\ar[dl]_{\mu}\\
G_0& \BG\\}$$

\begin{example}\rm \
When $G$ is a Lie group, one recovers (up to homotopy) the usual
classifying space of $G$ and the universal principal $G$-bundle
$EG\rightarrow BG$. More generally, for the groupoid $G \ltimes M$
associated to an action of $G$ on $M$, $B(G \ltimes M)$ is a model
for the homotopy quotient
\begin{equation*}
B(G \ltimes M) \cong M_G= (EG\times M)/G.
\end{equation*}
\end{example}

\subsection{The Bott-Shulman-Stasheff complex} In general, the geometric
realization of a simplicial manifold $X_\bullet$ is infinite
dimensional and in particular, it is not a manifold. However, there
is a De Rham theory that allows one to compute the cohomology of the
geometric realization $|| X_{\bullet} ||$ with real coefficients
using differential forms. Given a simplicial manifold $X_{\bullet}$
the Bott-Shulman-Stasheff complex \cite{Bott2}, denoted $\Omega(X_{\bullet})$, is the
double complex
$$
\xymatrix{
 \vdots& \vdots& \vdots & \\
 \Omega^2(X_0)\ar[u]\ar[r]^{\delta} &\Omega^2(X_1)\ar[u]^{\ud}\ar[r]^{\delta}
 &\Omega^2(X_2)\ar[u]^{\ud}\ar[r]^{\delta} & \dots\\
 \Omega^1(X_0)\ar[u]^{\ud}\ar[r]^{\delta} &\Omega^1(X_1)\ar[u]^{\ud}\ar[r]^{\delta}
 &\Omega^1(X_2)\ar[u]^{\ud}\ar[r]^{\delta} & \dots\\
 \Omega^0(X_0)\ar[u]^{\ud}\ar[r]^{\delta} &\Omega^0(X_1)\ar[u]^{\ud}\ar[r]^{\delta}
  &\Omega^0(X_2)\ar[u]^{\ud}\ar[r]^{\delta} & \dots\\
}$$ where the vertical differential is just the De Rham differential
and the horizontal differential $\delta$ is given by the simplicial
structure,
\begin{equation*}
\delta=\sum_{i=0}^{p+1}(-1)^id_i^*.
\end{equation*}
The total complex of $\Omega(X_{\bullet})$ is the De Rham model for
the cohomology of $||X_{\bullet}||$. We will also consider the
normalized Bott-Shulman-Stasheff complex of $X_{\bullet}$, denoted
$\hat{\Omega}(X)$, which is the subcomplex of $\Omega(X_{\bullet})$
that consists of forms $\eta \in \Omega^q(X_p)$ such that
$s_i^*(\eta)=0$ for all $i=0,\dots ,p-1$. The inclusion
$\hat{\Omega}(X_{\bullet}) \to \Omega(X_{\bullet})$ induces an
isomorphism in cohomology.
\begin{theorem}{\bf{(Dupont, Bott, Shulman, Stasheff,\dots)}} There is a natural isomorphism
\begin{equation*}
H(\mathrm{Tot}(\Omega(X_{\bullet}))) \cong H(|| X_{\bullet} ||).
\end{equation*}
\end{theorem}

For a Lie groupoid $G$ we will write $\Omega(G_{\bullet})$ instead
of $\Omega(\nerve)$. Note that the Bott-Shulman-Stasheff complex
$\Omega(G_{\bullet})$ provides us with an explicit model computing
$H^{\bullet}(BG)$. However, it is rather big and unsatisfactory
compared with the infinitesimal models available for Lie groups. We
would like to emphasize another aspect of the Bott-Shulman-Stasheff complex.
It is the natural place on which several geometric structures live.
The best example is probably that of multiplicative forms. We first
recall the definition (see, for instance, \cite{BUR}).

\begin{definition}
A multiplicative $k$-form on a Lie groupoid $G$ is a $k$-form
$\omega \in \Omega^k(G)$ satisfying
\begin{equation*}
d^{*}_{1}(\omega)=d^{*}_{0} (\omega)+d^{*}_{2}(\omega).
\end{equation*}
Given $\phi\in \Omega^{k+1}(M)$ closed, we say that $\omega$ is
relatively $\phi$-closed if $d\omega= s^*\phi- t^*\phi$.
% A relative multiplicative $k$-form on a Lie groupoid $G$ is a pair $(\omega, \phi)$ consisting
% of a $k$-form $\omega \in \Omega^k(G)$ and a $(k+1)$-form $\phi \in \Omega^{k+1}(M)$
% with the property that $\omega +\phi$ is a cocycle in the
% Bott-Shulman-Stasheff complex. Equivalently, $\omega$ is multiplicative in the sense that
% \begin{equation*}
% {d^*}_1(\omega)={d^*}_0 (\omega)+d^*_2(\omega),
% \end{equation*}
% $\phi$ is closed and $\omega$ is relatively $\phi$-closed in the sense that
% $d\omega= s^*\phi- t^*\phi$.
\end{definition}

In terms of the Bott-Shulman-Stasheff complex, the conditions appearing in
the previous definition can be put together into just one: $\omega
+\phi$ is a cocycle in the Bott-Shulman-Stasheff complex of $G$.

\begin{example}\label{IM} \rm \ With this terminology, a symplectic groupoid is
a Lie groupoid $G$ endowed with a symplectic form $\omega$ which is
multiplicative. This corresponds to the case $k= 2$, $\phi= 0$ in
the previous definition. Symplectic groupoids arise in Poisson
geometry, the global geometry of a Poisson manifold is encoded in a
topological groupoid which is a symplectic groupoid provided it is
smooth. In turn, smoothness holds under relatively mild topological
conditions. The case $k =2$ and $\phi$-arbitrary arises from various
generalizations of Poisson geometry which, in turn, show up in the
study of Lie-group valued momentum maps. With these motivations,
relatively closed multiplicative two forms have been intensively
studied in \cite{BUR} culminating with their infinitesimal
description which we now recall. Given a Lie algebroid $A$ over $M$
and a closed $3$-form $\phi$  on $M$, an IM (infinitesimally
multiplicative) form on $A$ relative to $\phi$ is, by definition, a
bundle map
\[ \sigma: A\rmap T^*M, \]
satisfying
\begin{eqnarray*}
\langle\sigma(\alpha),\rho(\beta)\rangle  & = & -\langle\sigma(\beta),\rho(\alpha)\rangle,\\
\sigma([\alpha, \beta])  & = & L_{\rho(\alpha)}(\sigma(\beta))-
                                      L_{\rho(\beta)}(\sigma(\alpha))+
                                      d\langle\sigma(\alpha),\rho(\beta)\rangle + i_{\rho(\alpha\wedge \beta}(\phi),
\end{eqnarray*}
for all $\alpha, \beta \in \Gamma(A)$. Here
$\langle\cdot,\cdot\rangle$ denotes the pairing between a vector
space and its dual. If $A$ is the Lie algebroid of a Lie groupoid
$G$, then any multiplicative $2$-form $\omega$ on $G$ which is
closed relative to $\phi$ induces such a $\sigma$:
\[ \sigma(\alpha)= i_{\alpha}(\omega)|_{M}.\]
The main result of \cite{BUR} says that, if the $s$-fibers of $G$
are $1$-connected, then the correspondence $\omega\mapsto \sigma$ is
a bijection. The basic example of this situation comes from Poisson
geometry. The cotangent bundle $T^*M$ of a Poisson manifold $M$
carries an induced Lie algebroid structure and the identity map is
an IM form. If $T^*M$ is integrable and $\Sigma(M)$ is the (unique)
Lie groupoid with $1$-connected $s$-fibers integrating it, the
corresponding multiplicative two form $\omega$ is precisely the one
that makes $\Sigma(M)$ a symplectic groupoid.
\end{example}

\section{The Weil algebra}\label{the Weil algebra}

In this section we introduce and discuss Weil algebras in the
context of Lie algebroids.

Throughout the section, $A$ is a fixed Lie algebroid over $M$.
The Weil algebra of $A$, denoted $W(A)$, will be a bi-graded
differential algebra. An element $c\in W^{p, q}(A)$ is a sequence $c= (c_0, c_1, \ldots )$ of operators
that satisfy some compatibility relation. Before explaining what
each $c_i$ is, we want to emphasize that $c_0$ should be viewed as
the leading term of $c$, while the remaining terms $c_1, c_2,
\ldots$ should be viewed as correction terms for $c_0$. The leading
term $c_0$ is just an antisymmetric $\mathbb{R}$-multilinear map
\[ c_0: \underbrace{\Gamma(A)\times \ldots \times \Gamma(A)}_{p\ \textrm{times}} \rightarrow \Omega^q(M) .\]
As a general principle, the role of the higher order terms is to
measure the lack of $C^{\infty}(M)$-linearity of $c_0$. With this in
mind, one can often compute the higher terms from $c_0$.

\begin{definition}\label{def-W} A element in $W^{p, q}(A)$ is a sequence of operators $c= (c_0, c_1, \ldots )$
with
\[ c_i: \underbrace{\Gamma(A)\times \ldots \times\Gamma(A)}_{p-i\ \textrm{times}} \rightarrow \Omega^{q-i}(M; S^i(A^*)), \]
satisfying
\[ c_i(\alpha_1, \ldots, f\alpha_{p-i})= fc_i(\alpha_1, \ldots , \alpha_{p-i})- df \wedge \partial_{\alpha_{p-i}}(c_{i+1}(\alpha_1, \ldots , \alpha_{p-i-1})),\]
for all $f\in C^{\infty}(M)$, $\alpha_i\in\Gamma(A)$.
\end{definition}

Here we use the notation from the Appendix. In particular, for
$\alpha\in \Gamma(A)$, $\partial_{\alpha}: S^k(A^*)\rmap
S^{k-1}(A^*)$ is the partial derivative along $\alpha$. Also,
viewing elements of $\Omega(M; S(A^*))$ as polynomial functions on
$A$ with values in $\Lambda T^*M$, we use the notation:
\[ c_i(\alpha_1, \ldots , \alpha_{p-i}| \alpha):= c_i(\alpha_1, \ldots , \alpha_{p-i})(\alpha) \in \Omega(M) \ \ \ (\textrm{for}\ \alpha\in \Gamma(A)).\]

\begin{remark}\label{rk-body}\rm \
Suppose that  $c, c'$ are elements of  $\W^{p, q}(A)$. If $c_0=
c'_0$ then $c= c'$ provided $q\leq \textrm{dim}(M)$.
\end{remark}

\subsection{The DGA structure} We now discuss the differential graded algebra structure on
$W(A)$. First of all, as in the case of the
Weil algebra of a Lie algebra, the differential $d$ of $W(A)$ is a sum of two
differentials
\[ d= d^{\textrm{v}}+ d^{\textrm{h}}.\]

{\bf The vertical differential} $d^{\textrm{v}}$ increases $q$ and it is induced by the De
Rham differential on $M$ in the following sense. Given $c\in W^{p,
q}(A)$, the leading term of $d^{\textrm{v}}(c)$ is, up to a sign,
just the De Rham differential of the leading term of $c$:
\[ (d^{\textrm{v}}c)_0(\alpha_1, \ldots , \alpha_p| \alpha)= (-1)^pd_{DR}(c_0(\alpha_1, \ldots , \alpha_p| \alpha)).\]
The other components $(d^{\textrm{v}}c)_k$ ($k\geq 1$) can be found
by applying the general principle mentioned above, by looking at the
failure of $C^{\infty}(M)$-linearity. For instance, replacing
 $\alpha_{p}$ with $f\alpha_{p}$ in the previous formula, one finds the following formula for the next component
of $d^{\textrm{v}}c$:
\[ (d^{\textrm{v}}c)_1(\alpha_1, \dots ,\alpha_{p-1}| \alpha)= (-1)^{p-1}(d_{DR}(c_1(\alpha_1, \ldots , \alpha_{p-1}| \alpha))+ c_0(\alpha_1, \ldots , \alpha_{p-1}, \alpha| \alpha)).\]
Proceeding inductively, one can find the explicit formulas for all
the other components. The final result, which will be taken as the
complete definition of $(d^{\textrm{v}}c)$, is:
\[ (d^{\textrm{v}}c)_k(\alpha_1, \dots ,\alpha_{p-k}|\alpha)=
(-1)^{p-k} (d_{DR}(c_k(\alpha_1, \dots ,\alpha_{p-k}|\alpha))+
c_{k-1}(\alpha_1, \ldots, \alpha_{p-k}, \alpha| \alpha)).\]

{\bf The horizontal differential $d^{\textrm{h}}$} increases $p$ and, as above, it is induced
by the Koszul differential in the following sense. Given $c\in W^{p,
q}(A)$, the leading term of $d^{\textrm{h}}(c)$ is given by the
Koszul differential of the leading term of $c$, where we use
$\Omega(M; SA^*)$ as a representation of the Lie algebra $\Gamma(A)$
(see the Appendix):
\begin{eqnarray*}
(d^{\textrm{h}}c)_0(\alpha_1, \dots ,\alpha_{p+1})&=& \sum_{i<j}(-1)^{i+j} c_{0}
([\alpha_i,\alpha_j],\dots,\hat{\alpha_i},\dots,\hat{\alpha_j},\dots,\alpha_{p+1})+\\
&+&\sum_i(-1)^{i+1}L_{\rho(\alpha_i)}(c_{0}(\alpha_1,\dots,\hat{\alpha_i},\dots,\alpha_{p+1})).
\end{eqnarray*}
As above, replacing $\alpha_{p+1}$ with $f\alpha_{p+1}$ and applying
the general principle, one finds the formula for the next component
of $d^{\textrm{h}}c$:
\[ (d^{\textrm{h}}c)_1(\alpha_1, \dots ,\alpha_{p}| \alpha)= \delta(c_1)(\alpha_1, \ldots , \alpha_{p}|\alpha)+ (-1)^{p-1} i_{\rho(\alpha)}c_0(\alpha_1, \ldots , \alpha_{p}).\]
Proceeding inductively, one finds the explicit formulas for all the
other components. The final result, which will be taken as the
complete definition of $(d^{\textrm{h}}c)$, is:
\[ (d^{\textrm{h}}c)_k(\alpha_1, \dots ,\alpha_{p-k+1}|\alpha)= \delta(c_k)(\alpha_1, \dots ,\alpha_{p-k+1}|\alpha)+ (-1)^{p-k}i_{\rho(\alpha)}c_{k-1}(\alpha_1, \ldots , \alpha_{p-k+1}|\alpha).\]

\begin{remark}\rm \  Our signs were chosen so
that they coincide with the standard ones for Lie algebra actions.
Admittedly, they do not look very natural.
\end{remark}

{\bf The algebra structure} on $W(A)$ is the following. Given $c\in W^{p, q}(A)$, $c'\in W^{p',
q'}(A)$ we describe $cc'\in W^{p+p', q+q'}(A)$. The leading
term is
\[ (cc')_0(\alpha_1, \ldots , \alpha_{p+p'}| \alpha)= (-1)^{q p'}\sum \textrm{sgn}(\sigma) c_{0}(\alpha_{\sigma(1)}, \ldots, \alpha_{\sigma(p)}| \alpha) c^{'}_{0}(\alpha_{\sigma(p+ 1)}, \ldots, \alpha_{\sigma(p+ p')}| \alpha) ,\]
where the sum is over all $(p, p')$-shuffles. The other components
can be deduced, again, by applying the general principle we have
already used. The general formula also follows from the relation
with the Kalkman's BRST algebra (Proposition \ref{rel-Kal} below).

\begin{theorem} $d^{\textrm{v}}$, $d^{\textrm{h}}$ and the product are well defined, $d^{\textrm{v}}$, $d^{\textrm{h}}$ are derivations
and
\[ d^{\textrm{v}}d^{\textrm{v}}= 0, d^{\textrm{h}}d^{\textrm{h}}= 0, d^{\textrm{v}}d^{\textrm{h}}+ d^{\textrm{h}}d^{\textrm{v}}= 0.\]
In conclusion, $W(A)$ becomes a bigraded bidifferential algebra.
\end{theorem}

\begin{proof} The statement is a consequence of Proposition \ref{rel-Kal}.
\end{proof}

In order to shed some light into the formulas, we point out the
relationship with the infinite dimensional version of Kalkman's BRST
algebra (see the Appendix). We will consider $W(\mathfrak{g};
\Omega(M))$ applied to the Lie algebra
\[ \mathfrak{g}_A:= \Gamma(A)\]
acting on $M$ via the anchor map. We will use the canonical
inclusion
\[ W(A)\hookrightarrow W(\mathfrak{g}_A; \Omega(M))\]
to realize $W(A)$ as a subspace of Kalkman's complex. From the
explicit formulas, we deduce the following:

\begin{proposition}\label{rel-Kal}
The algebra $W(A)$ is a sub-algebra of $W(\mathfrak{g}_A;
\Omega(M))$. Moreover, the horizontal and the vertical differentials
of $W(\mathfrak{g}_A; \Omega(M))$ restrict to $W(A)$ and coincide
with the ones defined above.
\end{proposition}
\begin{proof}
As a vector space $W(A)$ lives inside of $W(\mathfrak{g}_A;
\Omega(M))$ and a simple computation shows that it is closed under
the product. The explicit formulas for the differentials in the two
algebras clearly coincide on $W(A)$. Thus one only needs to show
that the differentials preserve the $W(A)$. For the vertical
differential we compute:
\begin{eqnarray*}
(d^{\textrm{v}}c)_k(\alpha_1, \dots ,f\alpha_{p-k}|\alpha)&=&
(-1)^{p-k} (d_{DR}(c_k(\alpha_1, \dots ,\alpha_{p-k}|\alpha))\\
&&+c_{k-1}(\alpha_1, \ldots, f\alpha_{p-k}, \alpha| \alpha))\\
&=& (-1)^{p-k} (fd_{DR}(c_k(\alpha_1, \dots ,\alpha_{p-k}|\alpha))\\
&&+df\wedge (c_k(\alpha_1, \dots ,\alpha_{p-k}|\alpha))\\
&&+df \wedge
d_{DR}(\partial_{\alpha_{p-k}}c_{k+1}(\alpha_1,\dots,\alpha_{p-k-1}|\alpha))\\
&&+ fc_{k-1}(\alpha_1, \ldots, \alpha_{p-k}, \alpha| \alpha)\\
&& +df \wedge
\partial{\alpha_{p-k}}(c_k(\alpha_1,\dots,\alpha_{p-k-1},\alpha|
\alpha)))\\
&=&f(d^{\textrm{v}}c)_k(\alpha_1, \dots ,f\alpha_{p-k}|\alpha)\\
&&- df\wedge
(\partial_{\alpha_{p-k}}(d^{\textrm{v}}c)_{k+1}(\alpha_1,\dots,\alpha_{p-k-1}|\alpha)).
\end{eqnarray*}
The fact that the horizontal differentials preserves $W(A)$ follows
by a similar computation once we observe that

\begin{eqnarray*}
\delta(c_k)(\alpha_1,\dots,f\alpha_{p-k+1}|\alpha)&=&
f \delta(c_k)(\alpha_1,\dots,\alpha_{p-k+1}|\alpha)\\
&&-df\wedge
\partial_{\alpha_{p-k+1}}\delta(c_{k+1})(\alpha_1,\dots,\alpha_{p-k}|\alpha)\\
&&+(-1)^{p-k+1}df\wedge
i_{\rho(\alpha_{p-k+1})}c_k(\alpha_1,\dots,\alpha_{p-k}|\alpha).
\end{eqnarray*}

\end{proof}

\begin{remark}\rm \ The previous proposition can be taken as a definition of the differentials and the product on $W(A)$.
The converse is more interesting, Kalkman's formulas can be
recovered  from the De Rham and Koszul differentials by computing
the higher order terms.
\end{remark}

\begin{example} \rm \ When $A= \mathfrak{g}$ is a Lie algebra
one recovers the usual Weil algebra. Also, when $A=
\mathfrak{g}\ltimes M$, one recovers Kalkman's differentials.
\end{example}

\subsection{The Weil algebra in local coordinates} \rm \ Since all the operators involved are
local, it is possible to describe $W(A)$ in coordinates.

\begin{definition}
Let $(x_a)$ be local coordinates in a chart for $M$ on which there is a
trivialization $(e_i)$ of the vector bundle $A$. Over this chart, we
obtain the following algebra $W_{\textrm{flat}}(A)$. As a bigraded
algebra, it is the commutative bigraded algebra over the space of smooth functions
generated by elements $\partial^a$ of bidegree $(0, 1)$, elements
$\theta^i$ of bi-degree $(1, 0)$ and elements $\mu^i$ of bi-degree
$(1, 1)$.
\end{definition}
There is an isomorphism between $W_{\textrm{flat}}(A)$ and $W(A)$, over the trivializing chars, given by:
\begin{enumerate}
\item $\partial^a$ to $dx_a\in \Omega^1(M)= W^{0, 1}(A)$.
\item  $\theta^i$ to the duals of $e_i$, viewed as elements in $\Gamma(\Lambda^1A^*)= W^{1, 0}(A)$.
\item $\mu^i$ to the elements $\widehat{\mu}^i\in W^{1, 1}(A)$, where
$\widehat{\mu}^{i}$ is determined by the fact that
$\widehat{\mu}^{i}_{0}$ vanishes on the $e_i$'s, while
$\widehat{\mu}^{i}_{1}$ is the dual of $e_i$, viewed as an element
of $\Gamma(S^1A^*)$.
\end{enumerate}
The  map $W_{\textrm{flat}}(A)\rmap W(A)$ is an isomorphism. The
differentials can now be computed explicitly on generators and one
finds (compare with \cite{GS}):
\begin{eqnarray*}
d_{\textrm{flat}}^{\textrm{v}}(\partial^a) & =&  0,\\
d_{\textrm{flat}}^{\textrm{v}}(\theta^i) & =& \mu^i,\\
d_{\textrm{flat}}^{\textrm{v}}(\mu^i)&=& 0,\\
d_{\textrm{flat}}^{\textrm{h}}(\partial^a)&=& -  \rho^{a}_{i} \mu^i+  \frac{\partial \rho^{a}_{i}}{\partial x_b}\theta^i\partial^b,\\
d_{\textrm{flat}}^{\textrm{h}}(\theta^i)&=& - \frac{1}{2} c^{i}_{jk} \theta^j\theta^k,\\
d_{\textrm{flat}}^{\textrm{h}}(\mu^i)&=& -  c_{jk}^{i}\theta^j\mu^k+
\frac{1}{2}\frac{\partial c_{jk}^{l}}{\partial x_a}
\theta^j\theta^k\partial^a,
\end{eqnarray*}
where we use the Einstein summation convention, $\rho^{a}_{i}$ are
the coefficients of $\rho$ and $c^{i}_{jk}$ are the structure
functions of $A$. Namely,
\[ \rho(e_i)= \sum \rho^{a}_{i}\partial_a,\ \ [e_j, e_k]= \sum c^{i}_{jk} e_i.\]
Note that, on smooth functions:
\[d_{\textrm{flat}}^{\textrm{v}}(f)= \partial_{a}(f) \partial^a, \ d_{\textrm{flat}}^{\textrm{h}}(f)= \partial_a(f)\rho_{i}^{a}\theta^i.\]

\subsection{The Weil algebra using a connection} \rm \ A global version of the previous
remark is possible with the help of a connection $\nabla$ on the
vector bundle $A$ and produces a version $W_{\nabla}(A)$ of $W(A)$
depending on $\nabla$. As a bigraded algebra it is just:
\[ W^{p, q}_{\nabla}(A)= \bigoplus_{k} \Gamma(\Lambda^{q-k}T^*M\otimes S^{k}(A^*)\otimes \Lambda^{p-k}(A^*)).\]
However, the associated operators $d^{\textrm{h}}_{\nabla}$,
$d^{\textrm{v}}_{\nabla}$ acting on $W_{\nabla}(A)$ are more
involved and are computed  in \cite{Ari-Cra1}. Working with the
global $\nabla$, one can write down the explicit local formulas  for
$d^{\textrm{h}}_{\nabla}$, $d^{\textrm{v}}_{\nabla}$ on generators.
The resulting equations will be similar to the ones for
$W_{\textrm{flat}}(A)$, but they have rather non-trivial extra-terms
which involve the coefficients of the connection and two types of
curvature tensors. The explicit map
\[ I_{\nabla}: W_{\nabla}(A)\rmap W(A),\]
is defined as follows. It is the unique algebra map which is
$C^{\infty}(M)$-linear and has the properties:
\begin{itemize}
\item on $\Omega(M)$ and $\Gamma(\Lambda A^*)$, which are subspaces of both $W_{\nabla}(A)$ and $W(A)$,
$I_{\nabla}$ is the identity.
\item for $\xi\in \Gamma(S^1A^*)$, $I_{\nabla}(\xi)= \widehat{\xi}$, where
\[ \widehat{\xi}_0(\alpha)= -\xi(\nabla (\alpha)), \widehat{\xi}_1= \xi.\]
\end{itemize}

\begin{proposition}\label{algebras are the same}
$I_{\nabla}$ is an isomorphism of bigraded algebras.
\end{proposition}

\begin{proof}
 We view $I_{\nabla}$ as a map of sheaves. It suffices to
show that $I_{\nabla}$ is an isomorphism locally. We then use the
generators $\partial^a$, $\theta^i$ and $\mu^i$ as above. These
elements also belong to the Kalkman algebra $W(\mathfrak{g}_A,
\Omega(M))$ and the map $I_{\nabla}$ can be seen as a map from
$W_{\nabla}(A)$ into $W(\mathfrak{g}_A, \Omega(M))$ which leaves
$\partial^a$ and $\theta^i$ invariant, but which sends $\mu^i$ into
$\widehat{\mu}^i$. Since the Kalkman algebra is free commutative and
the map is injective on the generators, we conclude that
$I_{\nabla}$ is injective. Surjectivity is a consequence of the fact
that $W(A)$ is generated by the elements in $W^{0,0}(A)$,
$W^{1,0}(A)$, $W^{0,1}(A)$ and the map $I_{\nabla}$ is clearly a
bijection in those degrees. Finally, since the differentials are
derivations, it is enough to prove that they coincide in low degree,
and this is a simple check.
\end{proof}

\subsection{The Weil algebra and the adjoint representation} \rm \ Let us give now a short
summary of our paper \cite{Ari-Cra1} and explain the connection with
the Weil algebra. In order to be able to talk about the adjoint
representation of a Lie algebroid, one has to enlarge the category
$\textrm{Rep}(A)$ of (standard) representations and work in the
category $\textrm{Rep}^{\infty}(A)$ of representations up to
homotopy. Such representations, by their nature, serve as
coefficients for the cohomology of $A$. Underlying any object of
$\textrm{Rep}^{\infty}(A)$ there is a cochain complex $(E,
\partial)$ of vector bundles over $A$; the extra-structure present
on $(E, \partial)$ is a linear operation of $A$ on $E$, which is not
quite an action- but the failure is precisely measured and there are
higher and higher coherence conditions. For instance, for the
adjoint representation, the underlying complex is:
\begin{equation}\label{adjoint}
A\stackrel{\rho}{\rmap} TM
\end{equation}
with $A$ in degree zero, $TM$ in degree one and zero otherwise.
However, to give this complex the structures of a representation up
to homotopy, one needs to use a connection $\nabla$ on the vector
bundle $A$. The resulting object $\textrm{Ad}_{\nabla}\in
\textrm{Rep}^{\infty}(A)$ does not depend on $\nabla$ up to
isomorphism. Its isomorphism class is denoted by $\textrm{Ad}$. This
indicates in particular that the resulting cohomologies with
coefficients in $\textrm{Ad}_{\nabla}$ or other associated
representations (such as symmetric powers, duals etc) do not depend
on $\nabla$ and can be computed by an intrinsic complex. The rows
$(W^{\bullet, q}(A), d^{\textrm{h}})$ of the Weil algebra are the
intrinsic complexes computing the cohomology of $A$  with
coefficient in $S^q(\textrm{Ad}^*)$:
\begin{equation}
\label{Sadjoint} H(W^{\bullet, q}(A))\cong H(A; S^q(\textrm{Ad}^*)).
\end{equation}
From this description it immediately follows that the cohomology of
$W(A)$ is isomorphic to the cohomology of $M$- which should be
expected because the fibers of the map $EG\rmap M$ are contractible.

\begin{example}[Multiplicative forms] \rm \ Closed multiplicative forms on
groupoids are related to homogeneous cocycles of the Weil algebra. To
illustrate this, let $A$ be the Lie algebroid of a Lie groupoid $G$
over $M$. Then, any $2$-form $\omega\in \Omega^2(G)$ induces an
element $c\in W^{1, 2}(A)$ with leading term
\[ c_0: \Gamma(A)\rmap \Omega^2(M), c_0(|\alpha )= L_{\alpha}(\omega)|_{M},\]
where $\alpha\in \Gamma(A)$ is identified with the induced right
invariant vector field on $G$ and we use the inclusion
$M\hookrightarrow G$ as units. The other component, $c_1\in
\Omega^1(M; S^1A^*)$, is given by
\[ c_1( |\alpha)= -i_{\alpha}(\omega)|_{M}.\]
When $\omega$ is closed $c$ is $d^{\textrm{v}}$ closed and when
$\omega$ is multiplicative $c$ is $d^{\textrm{h}}$-closed. This is
an instance of the Van Est map that will be explained in the next
section.
\end{example}

\begin{example}[IM forms] \rm \
In turn, $(1, 2)$ cocycles on the Weil algebra of a Lie algebroid
$A$ can be identified with the IM forms on $A$ (see Example \ref{IM}
in the case when $\phi= 0$). To see this, we first remark that an
element $c\in W^{1, 2}(A)$ which is $d^{\textrm{v}}$-closed is
uniquely determined by its component $c_1$, which we interpret as a
bundle map $A\rmap T^*M$ as before and denote it by $\sigma$.
Indeed,  the condition $(d^{\textrm{v}}c)_1= 0$ gives us
\[ c_0(|\alpha  )= -d_{DR}(\sigma(\alpha) ).\]
If $c$ is also $d^{\textrm{h}}$-closed  one has in particular that
$(d^{\textrm{h}}c)_2= 0$ and $(d^{\textrm{h}}c)_1= 0$. These two
conditions coincide with the conditions for $\sigma$ to be an IM
form (see Example \ref{IM}). One can check directly that,
conversely, these conditions also imply $(d^{\textrm{h}}c)_0= 0$.

As a conclusion of the last two examples, the correspondence between
multiplicative two-forms on groupoids and $IM$-forms on algebroids
described in Example (\ref{IM}) factors through the Weil algebra.
This will be generalized to arbitrary forms on the nerve of $G$ in
the next section. We will show that the main result of \cite{BUR}
can be derived as a consequence of a general Van Est isomorphism
theorem.
\end{example}

We now explain a version of the Weil algebra with coefficients which
will be used in the proof of our main theorem. The coefficients are
the generalizations of $\mathfrak{g}$-DG algebras to the context of
Lie algebroids.

\begin{definition} Given a Lie algebroid $A$ over $M$, an $A$-DG algebra
is a $DG$-algebra $(\mathcal{A}, d)$ together with
\begin{itemize}
\item a structure of $\Gamma(A)$-DG algebra, with Lie derivatives and
interior products denoted by $L_{\alpha}$ and $i_{\alpha}$,
respectively.
\item a graded multiplication $\Omega(M)\otimes \mathcal{A}\rmap \mathcal{A}$ which makes
$(\mathcal{A}, d)$ a DG-algebra over the De Rham algebra $\Omega(M)$
and which is compatible with $L_{\alpha}$ and $i_{\alpha}$ (i.e. it
is a map of $\Gamma(A)$-DG algebras)
\end{itemize}
such that
\[ i_{f\alpha}(a)= f i_{\alpha}(a), \ \ L_{f\alpha}(a)= fL_{\alpha}(a)+ (df) i_{\alpha}(a),\]
for all $\alpha\in \Gamma(A)$, $f\in C^{\infty}(M)$, $a\in
\mathcal{A}$.
\end{definition}

Given such an $A$-DG algebra, we define $W(A; \mathcal{A})$ as
follows. An element $c\in W^{p, q}(A; \mathcal{A})$ is a sequence
$(c_0, c_1, \ldots )$ where
\[ c_i: \underbrace{\Gamma(A)\times \ldots \times \Gamma(A)}_{p-i\ \textrm{times}}\times
\underbrace{\Gamma(A)\times \ldots \times \Gamma(A)}_{i\
\textrm{times}} \rightarrow \mathcal{A}^{q-i} ,\]
\[ (\alpha_1, \ldots , \alpha_{p-i}, \alpha_{p-i+1}, \ldots , \alpha_p)\mapsto c_i
(\alpha_1, \ldots , \alpha_{p-i}| \alpha_{p-i+1}, \ldots
,\alpha_{p}) \] is $\mathbb{R}$-multilinear and antisymmetric on
$\alpha_1, \ldots , \alpha_{p-i}$ and is $C^{\infty}(M)$-multilinear
and symmetric on $\alpha_{p-i+1},$ $\ldots$ , $\alpha_p$; moreover,
$c_i$ and $c_{i+1}$ are required to be related as in Definition
\ref{def-W}. As before, $W(A; \mathcal{A})$ sits inside Kalkman's
$W(\mathfrak{g}_A; \mathcal{A})$ and we use this inclusion to induce
the algebra structure and the two differentials on $W(A;
\mathcal{A})$.

\begin{example} \label{ex-ADG} The basic example of an $A$-DG algebra is the De Rham complex of $M$, in which case we recover $W(A)$. More generally,
if $A$ acts on a space $P\stackrel{\mu}{\rmap} M$ over $M$,
$\Omega(P)$ has the structure of an $A$-DG algebra. In this case
$L_{\alpha}$ and $i_{\alpha}$ are the usual Lie derivatives and
interior products with respect to the vector fields on $P$ induced
by $\alpha$, while the $\Omega(M)$-module structure is
\[ \Phi \cdot \omega= \mu^*(\Phi)\wedge \omega ,\]
for $\Phi\in\Omega(M)$ and $\omega\in \Omega(P)$.
\end{example}

\begin{remark}\label{the-trick} Consider an action of $A$ on $P\stackrel{\mu}{\rmap} M$ and the induced $A$-DG algebra structure on $\Omega(P)$.
Then, the algebra $W(A; \Omega(P))$ is isomorphic (as a bigraded
differential algebra) to $W(A\ltimes P)$, where $A\ltimes P$ is the
corresponding action Lie algebroid.
\end{remark}

%%%%%%%%%%%%%%%%%%%%%%%%%%%%%%%%%%%%%%%%%%%%%%%%%%%%%%%%%
%%%%%%%%%%%%%%%%%%%%%%%%%%%%%%%%%%%%%%%%%%%%%%%%%%%%%%%%%%%
%%%%%%%%%%%%%%%%%%%%%%%%%%%%%%%%%%%%%%%%%%%%%%%%%%%%%%%%
\section{The Van Est map}\label{section the Van Est map}
%%%%%%%%%%%%%%%%%%%%%%%%%%%%%%%%%%%%%%%%%%%%%%%%%%%%%%%%%%
%%%%%%%%%%%%%%%%%%%%%%%%%%%%%%%%%%%%%%%%%%%%%%%%%%%%%%%%%%%
%%%%%%%%%%%%%%%%%%%%%%%%%%%%%%%%%%%%%%%%%%%%%%%%%%%%%%%%%%%%

In this section we introduce the Van Est map which relates the
Bott-Shulman-Stasheff complex of a groupoid to the Weil algebra of its
algebroid.

Throughout this section, $G$ is a Lie groupoid over $M$ and $A$ is its Lie algebroid.
Any section $\alpha\in \Gamma(A)$ induces a vector field
$\alpha^p$ on each of the spaces $G_p$ of strings of $p$-composable
arrows.  Explicitly, for $g= (g_1, \ldots , g_p)\in G_p$ with $t(g)=
x$, $\alpha^{p}_{g}$ is the image of $\alpha_x\in A_x$ (i.e. in the
tangent space at $1_x$ of $s^{-1}(x)$) by the differential of the
map
\[ R_g: s^{-1}(x)\rmap G_p, \ \ a\mapsto ag:= (ag_1, g_2, \ldots , g_p).\]
The map $\alpha\rmap \alpha^p$ is nothing but the infinitesimal
action induced by the canonical right action of $G$ on $G_p$ (see
Subsection \ref{oids} and in particular Example \ref{action-nerve}).
When no confusion arises, we will denote the vector field $\alpha^p$
simply by $\alpha$.  The induced Lie derivative acting on
$\Omega(G_p)$, combined with the simplicial face map $s_0:
G_{p-1}\rmap G_p$ (which inserts a unit on the first place) induces
a map
\[ R_{\alpha}: \Omega^q(G_p)\rmap \Omega^q(G_{p-1}).\]
Intuitively, $R_{\alpha}(\omega)$ is the derivative on the first
argument along $\alpha$, at the units.

% Note that, as a map
% \[ \Gamma(A)\rmap \mathcal{X}(G_p) ,\]
% it gives an identification of $\Gamma(A)$ with the space of vector fields on $G_p$ which are tangent
% to the fibers of the simplicial map
%\[ d_0: G_p\rmap G_{p-1}, (g_1, g_2, \ldots, g_p)\mapsto (g_2, \ldots , g_p),\]
%and which are invariant with respect to the acion of $G$.

% Associated to any $\alpha\in \Gamma(A)$ there is an induced ``derivation operator''
% \[ R_{\alpha}: \Omega^q(G_p)\rmap \Omega^q(G_{p-1}),\]
% which takes the derivative on the first argument along $\alpha$, at the units.
% Explicitly,
% \[ R_{\alpha}(\omega)= s_{0}^{*}(L_{\alpha}(\omega)),\]
% where $s_0$ is the simplicial face map which inserts a unit on the first position.

\begin{proposition}\label{proposition Van Est}
Let $G$ be a Lie groupoid over $M$ with Lie algebroid $A$. For any
normalized form in the Bott-Shulman-Stasheff complex of $G$, $\omega\in
\hat{\Omega}^q(G_p)$, the map
\[ \underbrace{\Gamma(A)\times \ldots \times \Gamma(A)}_{p\ \textrm{times}} \rmap \Omega^q(M),\]
\[ (\alpha_1, \ldots , \alpha_p)\mapsto  (-1)^{\frac{p(p+1)}{2}} \sum_{\sigma\in S_p} \textrm{sgn}(\sigma) R_{\alpha_{\sigma(1)}}\ldots R_{\alpha_{\sigma(p)}} (\omega)\]
is the leading term of a canonical element $V(\omega)\in W^{p,
q}(A)$ induced by $\omega$. Moreover, the resulting map
\begin{equation*}
V :\hat{\Omega}^q(G_p) \to W^{p,q}(A)
\end{equation*}
is compatible with the horizontal and the vertical differentials in
the sense that
\begin{eqnarray}
V\ud&=&(-1)^p d^{\textrm{v}} V,\\
V \delta&=& d^{\textrm{h}} V.
\end{eqnarray}
\end{proposition}

\begin{remark}[More standard Van Est maps]\rm\
The standard Van Est map for a Lie groupoid $G$ relates the
differentiable cohomology $H^{\bullet}_{d}(G)$ with the cohomology
$H^{\bullet}(A)$ of the associated Lie algebroid. These cohomologies
can be identified in our framework as follows. $H^{\bullet}_{d}(G)$
is the cohomology of the first row $\Omega^{0}(G_{\bullet})$ of the
Bott-Shulman-Stasheff complex of $G$. On the other hand $H^{\bullet}(A)$ is
the cohomology of the first row $W^{\bullet, 0}(A)$ of the Weil
algebra. Our Van Est map extends the ordinary one to a map of double
complexes.
\end{remark}

As in the discussions in the previous section, one can heuristically
derive all the components of $V(\omega)$ out of the formula for the
leading term. However, strictly speaking we have to specify the
higher order terms for $V(\omega)$ to be well defined. To achieve
this, we need an operation similar to $R_{\alpha}$, but which uses
interior products instead of Lie derivatives:
\[ J_{\alpha}: \Omega^q(G_p)\rmap \Omega^{q-1}(G_{p-1}),\ \ J_{\alpha}(\omega):= s_{0}^{*}(i_{\alpha}(\omega)).\]

The component $\V(\omega)_i$ evaluated on sections of $\Gamma(A)$,
\[ V(\omega)_i(\alpha_1, \ldots , \alpha_{p-i}| \alpha)\in \Omega^{q-i}(M; S^{i}(A^*))\]
will be a sum in which each term arises by applying the operators
$R_{\alpha_k}$ $p-i$ times and $J_{\alpha}$ $i$ times in all
possible ways, with the appropriate sign. The summation is over all
permutations $\sigma\in S_p$ such that
\[ \sigma^{-1}(p-i+1) < \ldots < \sigma^{-1}(p-1) < \sigma^{-1}(p) .\]
We denote by $S_p(i)$ the set of all such permutations. For each
$\sigma\in S_p(i)$, we consider the expression
\[ V(\omega)_{i}^{\sigma}(\alpha_1, \ldots , \alpha_{p-i}| \alpha) := (-1)^iD_1\ldots D_p(\omega)\]
where the ordered sequence $D_1, \ldots , D_p$ is obtained as
follows. One starts with the sequence
\[ R_{\alpha_{\sigma(1)}}, \ldots , R_{\alpha_{\sigma(p)}} \]
and one replaces $R_{\alpha_{k}}$ by $J_{\alpha}$ whenever $k\in
\{p-i+1, \ldots , p\}$. With these conventions we define
\[ V(\omega)_i= (-1)^{\frac{p(p+1)}{2}}\sum_{\sigma\in S_p(i)} \textrm{sgn}(\sigma) V(\omega)_{i}^{\sigma}.\]

\begin{proof} (of Proposition \ref{proposition Van Est})
We first point out the following properties of the operators
$R_{\alpha}$ and $J_{\alpha}$, which follow immediately from similar
properties of the operators $L_{\alpha}$ and $i_{\alpha}$:
\begin{align}
&R_{\alpha}=J_{\alpha}\ud +\ud J_{\alpha}
\label{homotopy},\\
&R_{\alpha}(\eta \omega)=R_{\alpha}(\eta) s_0^*(\omega)+s_0^*(\eta)
R_{\alpha}(\omega) \label{r with product},\\
&J_{\alpha}(\eta \omega)=J_{\alpha}(\eta) s_0^*(\omega)+(-1)^q
s_0^*(\eta) J_{\alpha}(\omega) \label{jota with product},\\
&R_{f \alpha}(\eta)=d(f)J_{\alpha}(\eta)+f R_{\alpha}(\eta)\label{r takes f},\\
&J_{f \alpha}=f J_{\alpha} \label{jota takes f}.\\
\end{align}
Next, $R_{\alpha}$ and $J_{\alpha}$ interact with the degeneracy
maps $s_i$ as follows:
\begin{align}
&s_j^* J_{\alpha}=J_{\alpha}s_{j+1}^*\label{j with s},\\
&s_j^* R_{\alpha}=R_{\alpha}s_{j+1}^*\label{r with s}.
\end{align}
The second equation follows formally from the first one and formula
(\ref{homotopy}). The first one follows from the simplicial
relations and the equation
\begin{equation}
\label{help} s_{j+1}^*i_{\alpha^q}=i_{\alpha^{q-1}}s_{j+1}^*.
\end{equation}
In order to prove this last equation it is enough to evaluate it on
a one form $\omega \in \Omega^1(G_q)$. We will use the formula
\begin{equation*}
(ds_{j+1})_{g}(\alpha^{q-1})= \alpha_{s_{j+1}(g)}^q,
\end{equation*}
which follows from the definition of $\alpha^q$ and the fact that
$s_{j+1}R_g= R_{s_{j+1}(g)}$. Indeed, one simply computes:
\begin{eqnarray*}
i_{\alpha^{q-1}}s_{j+1}^*(\omega)_g&=&s_{j+1}^*(\omega)(\alpha^{q-1}_{g})\\
&=&\omega (ds_{j+1})_g(\alpha^{q-1}_{g}))\\
&=&\omega (\alpha^{q}_{s_{j+1}(g)})   \\
&=&s_{j+1}^*i_{\alpha^q}(\omega)_{g}.
\end{eqnarray*}
In particular, the equations above imply that $R_{\alpha}$ and
$J_{\alpha}$ preserve the normalized subcomplex. We will also use
the $\Omega(M)$-module structure on $\Omega(G_p)$ given by
\[ \Phi \omega= t^*(\Phi)\wedge \omega .\]
As a consequence of the previous formulas we have:
\[ R_{\alpha}(\Phi \omega)= \Phi R_{\alpha}(\omega), \ J_{\alpha}(\Phi \omega)=
(-1)^{\textrm{deg}(\Phi)} \Phi J_{\alpha}(\omega),\] for all
$\Phi\in \Omega(M), \omega\in \hat{\Omega}(G_{\bullet})$. From these
and (\ref{r takes f}) and (\ref{jota takes f}), it immediately
follows that the components $V(\omega)_i$ satisfy the desired
$C^{\infty}(M)$-linearity in the symmetric variables while on the
other variables we obtain the equation which expresses the relation
between $V(\omega)_i$ and $V(\omega)_{i+1}$. In other words,
$V(\omega)$ does belong to $W(A)$.

We will use the following remark on the functoriality of the Van Est
map. Given an action of $G$ on a space $P\stackrel{\mu}{\rmap} M$,
there is an action groupoid $G\ltimes P$ over $P$ with associated
Lie algebroid $A\ltimes P$. The pull-back from $M$ to $P$ induces
inclusions of the Weil algebra of $A$ and of the Bott-Shulman-Stasheff
complex of $G$, into the ones corresponding to $A\ltimes P$ and
$G\ltimes P$, respectively (see also Remark \ref{the-trick}), which
is compatible with the Van Est map:
\[ \xymatrix{
W(A) \ar[r]^-{V}\ar[d]^-{\textrm{incl}}  & \hat{\Omega}(G) \ar[d]^-{\textrm{incl}} \\
W(A\ltimes P) \ar[r]^-{V} & \hat{\Omega}(P\ltimes G) }\] Also, the
inclusion maps are clearly compatible with the vertical and the
horizontal differentials. We will use this diagram in order to
simplify the proof of the compatibility of $V$ with the
differentials. For instance, take $\omega\in \hat{\Omega}^q(G_p)$
and $q\leq \textrm{dim}(M)$. In order to prove that
\[ V(\ud(\omega))= (-1)^p d^{\textrm{v}}(V(\omega)),\]
it suffices to show that their leading terms coincide (see Remark
\ref{rk-body}). However, using a $G$-space $P$ with the dimension of
$P$ big enough (for a fixed $\omega$!), the previous diagram shows
that all we have to show is that
\[ V(\ud(\omega))_0= (-1)^p d^{\textrm{v}}(V(\omega))_0,\]
for all algebroids and all $\omega$'s. In turn, this formula follows
immediately from the definition of $d^{\textrm{v}}(\omega)_0$ and
the fact that the operations $R_{\alpha}$ commute with De Rham
differentials.

For the compatibility of $V$ with the horizontal differentials, we
will use the following formulas.
\begin{align}
R_{\alpha}d_i^*=\begin{cases}\label{r with d}
d_{i-1}^*R_{\alpha}&\text{ if } i>1,\\
L_{\alpha}& \text{ if } i=1,\\
0 & \text{ if } i=0.\\
\end{cases}\\
R_{\alpha}L_{\beta}- R_{\beta}L_{\alpha}= R_{[\alpha, \beta]}.
\label{commute}
\end{align}
Equations (\ref{r with d}) follow from the simplicial equations and
the following formula, which can be proven in the same way in which
(\ref{help}) was proved:
\begin{equation*}
i_{\alpha^{q+1}}d_{i}^*=\begin{cases}
{d_i}^*i_{\alpha^q}&\text{ if } i>0,\\
0 & \text{ if } i=0.\\
\end{cases}
\end{equation*}
Equation (\ref{commute}) follows immediately from $[L_{\alpha},
L_{\beta}]= L_{[\alpha, \beta]}$. We now prove the compatibility
with the horizontal differentials. As before, it suffices to show
that
\begin{equation}
\label{to-prove} V(\delta(\omega))_0= d^{\textrm{h}}(V(\omega))_0.
\end{equation}
Assume that  $\omega\in \Omega^q(G_{p-1})$ and we evaluate the
right hand side on $(\alpha_1, \ldots , \alpha_p)$. We have two
types of terms. The first type is
\[ \sum_{i= 1}^{p} (-1)^{i+1}L_{\alpha_i}(V(\omega)(\alpha_1, \ldots , \widehat{\alpha_i}, \ldots , \alpha_p)).\]
Writing out $V$ (for each $i$ fixed) we get a sum over permutations
$\sigma_0$ of $1, \ldots , \hat{i}, \ldots , p$. To the pair $(i,
\sigma_0)$ corresponds the permutation $\sigma= (i, \sigma_0(1),
\ldots )\in S_p$. Note that the number $\tau(\sigma)$ of
transpositions of $\sigma$ equals to $i-1+ \tau(\sigma_0)$, so the
sum above equals to
\begin{equation}
\label{inter-1} \sum_{i= 1}^{p} (-1)^{\frac{p(p-1)}{2}}
\textrm{sgn}(\sigma)L_{\alpha_{\sigma(1)}}(R_{\alpha_{\sigma(2)}}\ldots
R_{\alpha_{\sigma(p)}}\omega).
\end{equation}
The other term is
\[ \sum_{i< j} (-1)^{i+j} (V\omega)([\alpha_i, \alpha_j], \alpha_1, \ldots , \widehat{\alpha_i}, \ldots , \widehat{\alpha_j}, \ldots , \alpha_p).\]
Again, for each $i$ and $j$, writing out $V$ we get a sum over
permutations $\sigma_1$ of the list $$0, 1, \ldots , \hat{i}, \ldots
, \hat{j}, \ldots , p,$$ where $0$ is used to index the position of
$[\alpha_i, \alpha_j]$. To $(i, j, \sigma_1)$ we associate
\begin{itemize}
\item a number $k\in \{1, \ldots , p-1\}$ defined by the condition that $0$ is on the $(p-k)^{\textrm{th}}$-position of $\sigma_1$.
\item a permutation $\sigma \in S_p$ which is obtained from $\sigma_1$ by inserting $i$ on the $(p-k)^{\textrm{th}}$ place,
and $j$ on the $(p-k+1)^{\textrm{th}}$ (so that $\sigma(p-k)= i$,
$\sigma(p-k+1)= j$ and the ordered sequence $\sigma(1), \sigma(2),
\ldots$ from which $i$ and $j$ are deleted coincides with ordered
sequence $\sigma_1(0), \sigma_1(1), \ldots$ from which $0$ is
deleted.
\end{itemize}
Note that, modulo $2$, the number of transpositions of these
permutations satisfy:

\begin{eqnarray*}
\tau(\sigma_1) & = & \tau(\sigma(1), \ldots ,
\sigma(p-k-1), 0, \sigma(p-k+ 2), \ldots , \sigma(p)) \\
& = &  p-k +1+ \tau(\sigma(1), \ldots, \sigma(p-k-1), \sigma(p-k+2),
\ldots , \sigma(p)),
\end{eqnarray*}
\begin{eqnarray*}
\tau(\sigma) & = & \tau(\sigma(1), \ldots, \sigma(p-k-1), i, j, \sigma(p-k+2), \ldots , \sigma(p))
\equiv\\
& \equiv & \tau(i, j, \sigma(1), \ldots, \sigma(p-k-1), \sigma(p-k+2), \ldots , \sigma(p))
\\&&+ \tau (p-k, p-k+1, 1, 2, \ldots )\\
& = & \tau(i, j, \sigma(1), \ldots, \sigma(p-k-1), \sigma(p-k+2), \ldots , \sigma(p)) \\
& = & (i-1)+ (j-2)+ \tau(\sigma(1), \ldots, \sigma(p-k-1),
\sigma(p-k+2), \ldots , \sigma(p)).
\end{eqnarray*}
Thus
\[ (-1)^{i+j}\textrm{sgn}(\sigma_1)= (-1)^{p-k} \textrm{sgn}(\sigma) .\]
We conclude that the second term which arises from the right hand
side of (\ref{to-prove}) is:
\begin{equation}
\label{inter-2} (-1)^{\frac{p(p+1)}{2}}\sum_{k} \sum_{\sigma:
\sigma(p-k)< \sigma(p-k+1)} (-1)^{k}
\textrm{sgn}(\sigma)R_{\alpha_{\sigma(1)}} \ldots
R_{\alpha_{\sigma(p-k-1)}}R_{[\alpha_{\sigma(p-k)},
\alpha_{\sigma(p-k+1)}]} \ldots R_{\alpha_{\sigma(p)}}.
\end{equation}
The left hand side of (\ref{to-prove}) applied to $(\alpha_1, \ldots
, \alpha_p)$ is
\[ (-1)^{\frac{p(p+1)}{2}}\sum_{\sigma}\sum_{k= 0}^{p} \textrm{sgn}(\sigma) (-1)^k R_{\alpha_{\sigma(1)}}\ldots R_{\alpha_{\sigma(p)}} d_{k}^{*}\omega .\]
Using (\ref{r with d}), this is equal to
\[ (-1)^{\frac{p(p+1)}{2}}\sum_{k= 1}^{p} \textrm{sgn}(\sigma) (-1)^k R_{\alpha_{\sigma(1)}}\ldots R_{\alpha_{\sigma(p-k)}}L_{\alpha_{\sigma(p-k+1)}}\ldots R_{\alpha_{\sigma(p)}}.\]
When $k= p$ we obtain precisely (\ref{inter-1}). It remains to show
that the remaining terms give us (\ref{inter-2}). In that sum  (over
$\sigma$ and $k\leq p-1$) we distinguish two cases:
\begin{itemize}
\item $(k, \sigma)$ satisfies: $\sigma(p-k)< \sigma(p-k+1)$.
\item $(k, \sigma)$-satisfies: $\sigma(p-k)> \sigma(p-k+1)$.
\end{itemize}
Note that, using the transposition $\tau_k:= (p-k, p-k+1))$, we have
a bijection $(k, \sigma)\mapsto (k, \sigma \circ \tau_k)$ between
the first and second cases. Hence, both cases can be indexed by $(k,
\sigma)$ which satisfy $\sigma(p-k)< \sigma(p-k+1)$, but the second
case will produce terms of type:
\[ (-1)^{\frac{p(p+1)}{2}}(-\textrm{sgn}(\sigma)) (-1)^k R_{\alpha_{\sigma(1)}}\ldots R_{\alpha_{\sigma(p-k+1)}}L_{\alpha_{\sigma(p-k)}}\ldots R_{\alpha_{\sigma(p)}},\]
where we used that $\textrm{sgn}(\sigma\circ \tau_k)=
-\textrm{sgn}(\sigma)$. Putting together the two cases, we obtain:
\[ (-1)^{\frac{p(p+1)}{2}}(-1)^k\textrm{sgn}(\sigma)R_{\alpha_{\sigma(1)}}\ldots
(R_{\alpha_{\sigma(p-k)}}L_{\alpha_{\sigma(p-k+1)}}-
R_{\alpha_{\sigma(p-k+1)}}L_{\alpha_{\sigma(p-k)}})\ldots
R_{\alpha_{\sigma(p)}}, \] which combined with (\ref{commute}) gives
us precisely (\ref{inter-2}).
\end{proof}

% Google: kadernota UU
% STW
% ----------------------------------------

\section{The Van Est isomorphism}\label{section the Van Est isomorphism}

In this section we will prove the following Van Est isomorphism
theorem.

\begin{theorem}\label{Van Est}
Let $G$ be a Lie groupoid with Lie algebroid $A$ and $k$-connected
source fibers. The homomorphism induced in cohomology by the Van Est
map:
\begin{equation*}
V:H^p(\hat{\Omega}^q(G_{\bullet}))\rightarrow H^p(W^{\bullet,q}(A)),
\end{equation*}
is an isomorphism for $p \leq k$ and is injective for $p=k+1$.
\end{theorem}

\begin{remark}\rm \
When $q= 0$ one recovers the Van Est isomorphism of \cite{Cra1}. In
view of the isomorphism (\ref{Sadjoint}), the theorem gives an
isomorphism between $H^p(\Omega^q(G_{\bullet}))$ and $H^{p}(A;
S^{q}(\textrm{Ad}^*))$. When $G$ is a Lie group and $A=
\mathfrak{g}$ is a Lie algebra, this should be compared with the
result of Bott \cite{Bott} which gives an isomorphism between
$H^p(\Omega^q(G_{\bullet}))$ and the differentiable cohomology
$H^{p-q}_{d}(G; S^{q}\mathfrak{g}^*)$. In our coming paper
\cite{Ari-Cra2} we will show that the result of Bott holds for
arbitrary Lie groupoids.
\end{remark}

The proof of the theorem will be divided in two steps. First we will
prove that there is a homomorphism in cohomology which is an
isomorphism in the required degrees. Then we will prove that this
homomorphism coincides with the one induced by the Van Est map.

The first step is organized in the following co-augmented double
complex:
$$
\xymatrix{
\vdots &\vdots &\vdots &  \\
\Omega^q(G_2)\ar[r]^{d^*_0}\ar[u]^{\mathrm{d^{\textrm{h}}}}&
\W^{0,q}({\mathcal{F}}_2)\ar[u]^{\delta^{\textrm{v}}}
\ar[r]^{d^{\textrm{h}}}&\W^{1,q}({\mathcal{F}}_2)\ar[u]^{\delta^{\textrm{v}}}
\ar[r]&\dots\\
\Omega^q(G_1)\ar[r]^{d^*_0}\ar[u]^{\mathrm{d^{\textrm{h}}}}&
\W^{0,q}({\mathcal{F}}_1)\ar[u]^{\delta^{\textrm{v}}}
\ar[r]^{d^{\textrm{h}}}&W^{1,q}({\mathcal{F}}_1)\ar[u]^{\delta^{\textrm{v}}}
\ar[r]& \dots\\
\Omega^q(G_0)\ar[r]^{d^*_0}\ar[u]^{\mathrm{d^{\textrm{h}}}}&
\W^{0,q}({\mathcal{F}}_0)\ar[u]^{\delta^{\textrm{v}}}
\ar[r]^{d^{\textrm{h}}}&
W^{1,q}({\mathcal{F}}_0)\ar[u]^{\delta^{\textrm{v}}}
\ar[r]&\dots \\
 &  \W^{0,q}(A)\ar[u]^{\delta^{\textrm{v}}} \ar[r]^{d^{\textrm{h}}} &
\W^{1,q}(A)\ar[u]^{\delta^{\textrm{v}}}\ar[r] & \dots\\
}$$

Let us explain how this double complex is defined. As before, $G_k$
denotes the space of strings of $k$-composable arrows. Next,
$\mathcal{F}_k$ is the foliation on $G_{k+1}$ given by the fibers of
the map $d_0: G_{k+1}\rmap G_{k}$. We interpret $\mathcal{F}_k$ as
an integrable sub-bundle of $TG_{k+1}$ (namely the kernel of the
differential of $d_0$), hence also as a Lie algebroid over
$G_{k+1}$, with the inclusion as anchor. $W(\mathcal{F}_k)$ is the
Weil algebra of $\mathcal{F}_k$ and the $d^{\textrm{h}}$'s are the
corresponding horizontal differentials. We also define
$\mathcal{F}_{-1}:= A$.

The maps $d_{0}^{*}: \Omega^q(G_k)\rmap W^{0, q}(\mathcal{F}_{k})=
\Omega^q(G_{k+1})$ are given by the pull-back by $d_{0}$. To explain
$\delta^{\textrm{v}}$, we view $W(\mathcal{F}_k)$ as follows. First
of all, using the action of $G$ on $G_{k+1}$ (Example
\ref{action-nerve}) and the induced infinitesimal action of $A$ on
$G_{k+1}$, we have already mentioned that the associated Lie
algebroid $A\ltimes G_{k+1}$ can be identified with $\mathcal{F}_k$
(see Example \ref{action-gr-inf}). Hence, by Remark \ref{the-trick},
we have a canonical isomorphism
\[ W(\mathcal{F}_k) \cong W(A; \Omega(G_{k+1})),\]
where the left hand side is the Weil algebra with coefficients in
the $A$-DG algebra $\Omega(G_{k+1})$ associated to the action of $A$
on $G_{k+1}$ (see Example \ref{ex-ADG}). Since the simplicial maps
$d_i: G_{k+1}\rmap G_{k}$ are maps of $G$-spaces for $i\geq 1$, we
will have induced maps
\[ d_{i}^{*}: W(A; \Omega(G_{k}))\rmap W(A; \Omega(G_{k+1}))\]
which commute with $d^{\textrm{h}}$. We define
\[ \delta^{\textrm{v}}= \sum_{i\geq 1} (-1)^{i} d_{i}^{*}.\]
This completes the description of the double complex. Next, we claim
that the co-augmented columns of the double complex
$$
\xymatrix{0\ar[r]& W^{p,q}(A)\ar[r]^{\delta^{\textrm{v}}} &
W^{p,q}({\mathcal{F}}_0) \ar[r]^{\delta^{\textrm{v}}}&
W^{p,q}({\mathcal{F}}_1) \ar[r]^{\delta^{\textrm{v}}}&
 W^{p,q}({\mathcal{F}}_1)\ar[r]^{\delta^{\textrm{v}}}&\dots
}$$ are exact. This is a rather standard argument. Since
$\delta^{\textrm{v}}$ comes from a simplicial structure arising from
the nerve of $G$ by deleting the first face map ($d_0$ is not used
in the definition of $\delta^{\textrm{v}}$), the first degeneracy
map $s_0$ can be used to produce a contraction. More precisely,
since $s_{0}^{*}: \Omega(G_{k})\rmap \Omega(G_{k-1})$ respects the
$\Omega(M)$-module structure, it will induce a map:
\[ s_{0}^{*}: W(A; \Omega(G_{k}))\rmap W(A;\Omega(G_{k-1})).\]
The fact that $s_{0}^{*}$ is a contracting homotopy follows formally
from the simplicial identities:
\begin{eqnarray*}
\sigma^*_0\delta^{\textrm{v}}+\delta^{\textrm{v}}\sigma^*_0&=&\sum_{i=1}^{j+2}(-1)^{i+1}\sigma^*_0\delta^*_i+
\sum_{i=1}^{j+1}(-1)^{i+1}\delta^*_i\sigma^*_0\\
&=&Id+\sum_{i=2}^{j+2}(-1)^{i+1}\sigma^*_0\delta^*_i+
\sum_{i=1}^{j+1}(-1)^{i+1}\delta^*_i\sigma^*_0\\
&=&Id-\sum_{i=1}^{j+1}(-1)^{i+1}\delta^*_i\sigma^*_0+
\sum_{i=1}^{j+1}(-1)^{i+1}\delta^*_i\sigma^*_0=Id.
\end{eqnarray*}

This proves that the co-augmented columns of the double complex are
exact. The standard homological algebra of double complexes implies
that the map induced by the co-augmentation maps are isomorphisms:
\begin{equation*}
a:H(W^{\bullet,q}(A))\cong
H(Tot(W^{\bullet,q}({\mathcal{F}}_{\bullet}))).
\end{equation*}

Next, we look at the co-augmented rows
$$
\xymatrix{ 0\ar[r]&\Omega^q(G_{j})\ar[r]^{d^*_0}&
\W^{0,q}({\mathcal{F}}_j)\ar[r]& \W^{1,q}({\mathcal{F}}_j)\ar[r] & \ldots, \\
}$$ and we show that, under the assumption in the theorem, these are
exact up to degree $k$. Again, homological algebra implies that the
map induced by the co-augmentation of the columns
\begin{equation*}
b:H^p(\Omega^q(G_{\bullet}))\to
H^p(Tot(W^{\bullet,q}({\mathcal{F}}_{\bullet})),
\end{equation*}
is an isomorphism for $p<k+1$ and is injective for $p=k+1$. To prove
the (partial) acyclicity of the rows, we will use the following
lemma.

\begin{lemma}\label{vanishing for foliations}
Let $\mathcal{F}$ be a foliation given by the fibers of a submersion
with homologically $k$-connected fibers. Then, for $0<p \leq k$,
\begin{equation*}
H^p(W^{\bullet,q}({\mathcal{F}}))=0.
\end{equation*}
\end{lemma}
\begin{proof}
Here we will use the interpretation of the Weil algebra in terms of
representations up to homotopy given in Remark \ref{adjoint}. The
adjoint complex of a foliation is quasi-isomorphic to the complex
$\nu[1]$ which consists of the normal bundle $\nu= TM/\mathcal{F}$
concentrated in degree $1$. We now use the properties of
representations up to homotopy proved in \cite{Ari-Cra1}. We observe
that the representation up to homotopy $S^q(\textrm{Ad}^*)$ is
quasi-isomorphic to the ordinary representation $\Lambda^q(\nu^*)$.
Passing to cohomology and using the isomorphism (\ref{Sadjoint}), we
deduce that
\[ H^p(W^{\bullet,q}({\mathcal{F}}))\cong H^p({\mathcal{F}}, S^q(\nu ^*))=0\]
where the last equation is a direct application of Theorem $2$ from
\cite{Cra1}.
\end{proof}

We still have to show that the map
\[ a^{-1}\circ b: H^p(\Omega^q(G_{\bullet}))\rightarrow H^p(W^{\bullet,q}(A)),\]
which is an isomorphism in the desired degrees, is the same as the
Van Est map. More precisely we will show that, in cohomology,
\[ V= \pm a^{-1}\circ b \circ e ,\]
where $e$ is the isomorphism induced by the inclusion
$\hat{\Omega}(G_{\bullet})\hookrightarrow \Omega(G_{\bullet})$.
First we observe that $s^*_0$, the homotopy operator for the columns
of the double complex, gives a formula for the map $a^{-1}$.
Consider an element $c \in W^{j,q}(\mathcal{F}_p)$ such that
$\delta^{\textrm{v}}(c)=d^{\textrm{h}}(c)=0$. Then, chasing the
diagram we obtain that $a^{-1}(c)=(-1)^ps^*_0(d^{\textrm{h}}
s^*_0)^p$. We will use this formula to compute $a^{-1}\circ b \circ
e$. Take an element $\eta \in \hat{\Omega^q}(G_p)$ such that
$\delta(\eta)=0$ and compute:
\begin{eqnarray*}
a^{-1}\circ b \circ e(\eta)=(-1)^ps^*_0(d^{\textrm{h}}
s^*_0)^pd_0^*(\eta)=(-1)^p(s^*_0d^{\textrm{h}})^p(\eta).
\end{eqnarray*}
We claim that for each $0\leq l\leq p$:
\begin{equation*}
(s^*_0d^{\textrm{h}})^l(\eta)_0(\alpha_1,\dots,\alpha_l)=(-1)^{lq}\sum_{\lambda
\in S_l}(-1)^{|\lambda|}R_{\alpha_{\lambda_1}}\dots
R_{\alpha_{\lambda_l}}(\eta),
\end{equation*}
and also that $s^*_0((s^*_0d^{\textrm{h}})^l(\eta)_0(\alpha_1,\dots
,\alpha_l))=0$. We will prove our claim by induction on $l$. For
$l=0$ the claim is true because $\eta$ is normalized. We assume now
that the condition holds for $l-1$ and compute:
\begin{eqnarray*}
(s^*_0d^{\textrm{h}})^l(\eta)_0(\alpha_1,\dots,\alpha_l)
&=&s^*_0(d^{\textrm{h}}((s^*_0d^{\textrm{h}})^{l-1}(\eta)(\alpha_1,\dots,\alpha_l)\\
&=&s^*_0(\sum_{i=1}^{l}(-1)^{i+q+1}L_{\alpha_i}((s^*_0d^{\textrm{h}})^{l-1}(\eta))_0
(\alpha_1,\dots,\hat{\alpha_i},\dots,\alpha_l))\\
&=&(-1)^{q}(\sum_{i=1}^{l}(-1)^{i+1}R_{\alpha_i}((s^*_0d^{\textrm{h}})^{l-1}(\eta))_0
(\alpha_1,\dots,\hat{\alpha_i},\dots,\alpha_l))\\
&=&(-1)^{lq}\sum_{\lambda \in
S_l}(-1)^{|\lambda|}R_{\alpha_{\lambda_1}}\dots
R_{\alpha_{\lambda_l}}(\eta).
\end{eqnarray*}
In particular, for $l=p$ this means that $V=\pm a^{-1}\circ b \circ
e$. Since all $e,b, a^{-1}$ are isomorphisms in the required
degrees, this completes the proof.

\begin{corollary}
Let $G$ be a Lie groupoid with Lie algebroid $A$ and $k$-connected
source fibers. Then, the Van Est map $V:
H^p(\Omega(G_{\bullet}))\rightarrow H^p(\mathrm{Tot}(W(A)))$ is an
isomorphism for $p \leq k$ and is injective for $k=p+1$.
\end{corollary}

\begin{remark}
The isomorphism theorem \ref{Van Est} is the result one would expect
from a topological point of view. The map $V$ corresponds to the
projection $EG\rightarrow BG$, whose fibers are isomorphic to the
fibers of the source of $G$. In case this map is a fibration, the
Lerray-Serre spectral sequence gives isomorphisms in cohomology in
degrees less than the connectedness of the fiber.
\end{remark}

%%%%%%%%%%%%%%%%%%%%%%%%%%%%%%%%%%%%%%%%%%%%%%%%%%%%%%%
%%%%%%%%%%%%%%%%%%%%%%%%%%%%%%%%%%%%%%%%%%%%%%%%%%%%%%%
%%%%%%%%%%%%%%%%%%%%%%%%%%%%%%%%%%%%%%%%%%%%%%%%%%%%%%%
%%%%%%%%%%%%%%%%%%%%%%%%%%%%%%%%%%%%%%%%%%%%%%%%%%%%%%%

\section{Applications}\label{integrability}

In this section we discuss multiplicative forms from the point of
view of the Van Est isomorphism theorem. In particular, we will prove
Theorem \ref{xxx2} and Theorem \ref{xxx3} from the introduction.

We start with the following more precise version of Theorem \ref{xxx2},
which generalizes the main result of \cite{BUR} on integration of Dirac structures.

\begin{theorem}\label{cohomological integration1}
Let $G$ be a source simply connected Lie groupoid over $M$ with Lie
algebroid $A$ and let $\phi \in \Omega^{k+1}(M)$ be a closed form.
Then there is a one to one correspondence between:
\begin{itemize}
\item multiplicative forms $\omega \in \Omega^k(G)$ which are $\phi$-relatively closed.
\item $C^{\infty}(M)$-linear maps $\tau:\Gamma(A)\rightarrow \Omega^{k-1}(M)$
satisfying the equations:
\begin{eqnarray}
& i_{\rho(\beta)}(\tau(\alpha))  =  -i_{\rho(\alpha)}(\tau(\beta)),\label{mk-1}\\
& \tau([\alpha,\beta])  =  L_{\alpha}(\tau(\beta))-L_{\beta}(\tau(\alpha))
+d_{DR}(i_{\rho(\beta)}\tau(\alpha))+ i_{\rho(\alpha)\wedge
\rho(\beta)}(\phi).\label{mk-2}
\end{eqnarray}
for all $\alpha, \beta\in \Gamma(A)$.
\end{itemize}
The correspondence is given by:
\[ \tau(\alpha)= i_{\alpha}(\omega)|_{M},\]
where, on the right hand side, $\alpha\in \Gamma(A)$ is identified with the corresponding right invariant vector field on $G$ and
the restriction to $M$ makes use of the inclusion $M\hookrightarrow G$ as units.
\end{theorem}

First we show that the correspondence $\omega\mapsto \tau$ is
well-defined, i.e. $\tau$ satisfies the equations (\ref{mk-1}) and
(\ref{mk-2}). For that, we first remark that $\tau$ is precisely
$V(\omega)_1\in \Omega^{k-1}(M; S^1(A^*))$, viewed as a map
$\Gamma(A)\rmap \Omega^{k-1}$. Since $\omega + \phi$ is a cocycle in
the Bott-Shulman-Stasheff complex (see subsection \ref{oids}), it follows
that $V(\omega)+ \phi$ is a cocycle in the Weil algebra. The desired
equations for $\tau$ will then be implied by the following:

\begin{proposition}\label{int-pr} Given $\phi\in \Omega^{k+1}(M)$, $\sigma= (\sigma_0, \sigma_1)\in W^{1, k}(A)$,
$\sigma+ \phi$ is a cocycle in the Weil algebra if and only if:
\begin{enumerate}
\item $\phi$ is a closed form and $\sigma_1$ satisfies equations (\ref{mk-1}), (\ref{mk-2}).
\item $\sigma_0(\alpha)= i_{\rho(\alpha)}(\phi)- d( \sigma_1(\alpha))$.
\end{enumerate}
\end{proposition}

% \begin{proposition}\label{int-pr} Given a closed form $\phi\in \Omega^{k+1}(M)$, there is a 1-1 correspondence between
% \begin{enumerate}
% \item elements $\sigma\in W^{1, k}(A)$, $\phi \in \Omega^{k+1}(M)$ such that
% $sigma+ \phi$ is a cocycle in the Weil complex.
% \item $C^{\infty}(M)$-linear maps $\tau:\Gamma(A)\rightarrow \Omega^{k-1}(M)$
% satisfying the equations (\ref{mk-1}) and (\ref{mk-2}).
% \end{enumerate}
% This correspondence associates to $\sigma$ its 1-component $\sigma_1$, viewed as a map $\Gamma(A)\rightarrow \Omega^{k-1}(M)$.
% Its inverse, associates to $\tau$ the $\sigma$ with the following components:
% \[ \sigma_0(\alpha)= i_{\rho(\alpha)}(\phi)- d( \tau(\alpha)), \sigma_1(\alpha)= \tau(\alpha).\]
% \end{proposition}

\begin{proof}
The condition that $\sigma+ \phi$ is a cocycle means that
$d^{\textrm{v}}(\sigma)+ d^{\textrm{h}}(\phi)= 0$. But
\[ d^{\textrm{v}}(\sigma)_0(\alpha)= - d_{DR}(\sigma_0(\alpha)), (d^{\textrm{v}}\sigma)_1( | \alpha)= d_{DR}(\sigma_1(|\alpha))+ \sigma_0(\alpha) ,\]
\[ (d^{\textrm{h}}\phi)_0(\alpha)= L_{\rho(\alpha)}(\phi), (d^{\textrm{h}}\phi)_1( | \alpha)= -i_{\rho(\alpha)}(\phi), \]
hence we obtain the equations:
\[ \sigma_0(\alpha)= i_{\rho(\alpha)}(\phi)- d(\sigma_1(\alpha)), d(\sigma_0(\alpha))= L_{\rho(\alpha)}(\phi).\]
Since the second equation is obtained by applying $d_{DR}$ to the
first one and using that $\phi$ is closed, we only have to keep in
mind the first equation.

The other condition for $\sigma+ \phi$ to be a cocycle is
$d^{\textrm{h}}(\sigma)= 0$. We write the components:
\[ (d^{\textrm{h}}\sigma)_0(\alpha, \beta)= - \sigma_0([\alpha, \beta])+ L_{\rho(\alpha)}(\sigma_0(\beta))- L_{\rho(\beta)}(\sigma_0(\alpha)),\]
\[ (d^{\textrm{h}}\sigma)_1(\alpha| \beta)= L_{\rho(\alpha)}(\sigma_1(\beta))- \sigma_1([\alpha, \beta])+ i_{\rho(\beta)}(\sigma_0(\alpha)),\]
\[ (d^{\textrm{h}}\sigma)_2(|\alpha)= - i_{\rho(\alpha)}(\sigma_1(\alpha)).\]
 Clearly, $(d^{\textrm{h}}\sigma)_2= 0$ is equivalent to (\ref{mk-1}). Also, $(d^{\textrm{h}}\sigma)_1= 0$ is equivalent to
\[ \sigma_1([\alpha, \beta])= L_{\rho(\alpha)}(\sigma_1(\beta))- L_{\rho(\beta)}(\sigma_1(\alpha))+ di_{\rho(\beta)}(\sigma_1(\alpha))+ i_{\rho(\beta)}i_{\rho(\alpha)}(\phi),\]
which is equivalent to (\ref{mk-2}). Finally, a simple computation
shows that if $(d^{\textrm{h}}\sigma)_2= 0$ and
$(d^{\textrm{h}}\sigma)_1= 0$ then $(d^{\textrm{h}}\sigma)_0= 0$.
\end{proof}

Next, we prove that the correspondence $\omega\mapsto \tau$ in the
theorem is injective. Since $\tau= V(\omega)_1$, this part follows
from the following:

\begin{lemma} \label{l-uniq}
If $G$ is a Lie groupoid with connected source-fibers and $\omega\in
\Omega^k(G)$ is closed and multiplicative, then the following are
equivalent:
\begin{enumerate}
\item[(i)] $V(\omega)= 0$.
\item[(ii)] $V(\omega)_1= 0$.
\item[(iii)] $\omega_x= 0$ for all $x\in M$.
\item[(iv)] $\omega= 0$.
\end{enumerate}
\end{lemma}

\begin{proof}
Since $V(\omega)$ is a cocycle in the Weil algebra, Proposition
\ref{int-pr} tells us that $V(\omega)$ is determined by
$V(\omega)_1$, hence (i) and (ii) are equivalent. In turn, from the
definition of $V(\omega)_1$, (ii) means that
\[ \omega(\alpha_x, V^{2}_{x}, \ldots , V^{k}_{x})= 0\]
for all $x\in M$, $V^{k}_{x}\in T_xM$, $\alpha\in \Gamma(A)$, where
we identify $\alpha$ with the induced right invariant vector field
on $G$. In other words, $\omega_x$ is zero when applied to one
vector tangent to the $s$-fiber and $(k-1)$ vectors tangent to the
base. We have to show that this implies $\omega_x$ is zero when
applied to all vectors. But $T_xG$ splits as the sum of the tangent
space to the $s$-fiber and the tangent space of $M$ (both at $x$),
hence it remains to show that $\omega|M= 0$. But this follows
immediately from $\omega|_{M}= s_{0}^{*}d^{\textrm{h}}\omega$.
Finally, we have to show that (iii) implies $\omega= 0$. For this we
evaluate expressions of type
\begin{equation}
\label{omega-part} \omega_g(\alpha_g, V^{2}_{g}, \ldots , V^{k}_{g})
\end{equation}
for $g\in G$, $\alpha\in \Gamma(A)$, $V^{i}_{g}\in T_gG$ arbitrary.
To make use of the multiplicativity of $\omega$, we write
\[ \alpha_g= (dm)_{y, g}(\alpha_y, 0), v^{i}_{g}= (dm)_{y, g}((dt)_g(V^{i}_{g}), V^{i}_{g})\]
and we find that (\ref{omega-part}) is equal to
\[ \omega(\alpha_y, (dt)_g(V^{2}_{g}), \ldots , (dt)_g(V^{k}_{g})),\]
which, by assumption, is zero. We conclude that $i_{\alpha}(\omega)=
0$ for all $\alpha$ hence, since $\omega$ is also closed, it is
basic  with respect to the submersion $s: G\rmap M$. Since the
$s$-fibers are connected, we find $\theta$ on $M$ such that $\omega=
s^{*}\theta$. But $\omega|_{M}= 0$ implies $\theta= 0$ and then
$\omega= 0$.
\end{proof}

\begin{proof} (end of proof of theorem \ref{cohomological integration1}) Finally, we prove that the correspondence $\omega\mapsto \tau$ in
the theorem is surjective. Note that the case $k=2$ was proved in
\cite{BUR} and surjectivity was the most difficult part of the
proof. Given $\tau$, take $\sigma\in W^{1, k}(A)$ as in Proposition
\ref{int-pr} with $\sigma_1= \tau$. Since $\sigma$ is
$d^{\textrm{h}}$-closed, Theorem \ref{Van Est} implies that there
exist some multiplicative form $\omega' \in \Omega^p(G)$ such that
$V(\omega')=d^{\textrm{h}}(\theta)+\sigma$, for some $\theta \in
\Omega^k(M)$. It is then clear that $\omega= \omega'-\delta(\theta)$
is a multiplicative form satisfying $V(\omega)=\sigma$. In
particular, $V(\ud\omega+ \delta\phi)= d^{\textrm{v}}\sigma+
d^{\textrm{h}}\phi= 0$ and $d^{\textrm{v}}\omega+
d^{\textrm{h}}\phi$ is both multiplicative and closed. Using the
previous lemma, we find that $\omega$ is $\phi$-relatively closed.
This concludes the proof of Theorem \ref{cohomological
integration1}.
\end{proof}

Next, we discuss Theorem \ref{xxx3} from the introduction. We will
prove the following more precise statement.

\begin{theorem}\label{cohomological integration2} Let $G$ be a source simply connected Lie groupoid over $M$ with Lie
algebroid $A$ and let $\omega\in \Omega^k(G)$ be a closed
multiplicative $k$-form. Then there is a 1-1 correspondence between:
\begin{itemize}
\item $\theta\in \Omega^{k-1}(G)$ multiplicative satisfying $\ud(\theta)= \omega$.
\item $C^{\infty}(M)$-linear maps $l: \Gamma(A)\rmap \Omega^{k-2}(M)$ satisfying
\begin{eqnarray}
&i_{\rho(\beta)}(l(\alpha))=-i_{\rho(\alpha)}(l(\beta)),\label{mmk-1}\\
& c_{\omega}(\alpha, \beta)= - l([\alpha, \beta])+
L_{\rho(\alpha)}(l(\beta))- L_{\rho(\beta)}(l(\alpha))+
d_{DR}(i_{\rho(\beta)}l(\alpha)). \label{mmk-2}
\end{eqnarray}
\end{itemize}
where $c_{\omega}(\alpha, \beta)= i_{\rho(\alpha)\wedge \rho(\beta)}(\omega)|_{M}$. The correspondence is given by
\begin{equation*}
l(\alpha)= - i_{\alpha}(\theta)|_{M}.
\end{equation*}
\end{theorem}
From the previous two theorems we immediately deduce what is the infinitesimal 
data associated to
multiplicative forms on groupoids.

\begin{remark}\rm \ When $k= 2$, $\Omega^{k-2}(M)= C^{\infty}(M)$, so that (\ref{mmk-1}) is void
while (\ref{mmk-2}) simply becomes $c_{\omega}= \delta(l)$, where $\delta$ is the differential
of the DeRham complex $(\Omega(A), \delta)$ of $A$. Hence, in this case, we obtain the following:
if $\omega\in \Omega^2(G)$ is multiplicative and
closed then $c_{\omega}\in \Omega^2(A)$ is a cocycle and there is a
1-1 correspondence between $\theta\in \Omega^1(G)$ multiplicative
such that $\ud\theta= \omega$ and $l\in \Omega^1(A)$ satisfying
$\delta(l)= \omega$. This is the statement that appears in \cite{Cra2}.
\end{remark}
\begin{corollary}\label{cor:hen} Let $G$ be a source simply connected Lie groupoid over $M$ 
with Lie
algebroid $A$. Then there is a one to one correspondence between:
\begin{itemize}
\item multiplicative forms $\theta \in \Omega^k(G)$.
\item $C^{\infty}(M)$-linear maps $\tau:\Gamma(A)\rightarrow 
\Omega^{k}(M)$ and
$l: \Gamma(A)\rmap \Omega^{k-1}(M)$ satisfying the equations:
\begin{eqnarray}
&i_{\rho(\beta)}(l(\alpha))=-i_{\rho(\alpha)}(l(\beta)),\nonumber\\
& i_{\rho(\beta)}(\tau(\alpha))= - l([\alpha, \beta])+
L_{\rho(\alpha)}(l(\beta))- L_{\rho(\beta)}(l(\alpha))+
d_{DR}(i_{\rho(\beta)}l(\alpha)). \nonumber\\
& \tau([\alpha,\beta])  =  L_{\alpha}(\tau(\beta))-L_{\beta}(\tau(\alpha))
+d_{DR}(i_{\rho(\beta)}\tau(\alpha)),\nonumber
\end{eqnarray}
for all $\alpha, \beta\in \Gamma(A)$.
\end{itemize}
\end{corollary}

\begin{proof} A multiplicative $k$-form $\theta$ is determined by the following data:
\begin{enumerate}
\item A closed multiplicative $(k+1)$-form $\omega$,
\item A multiplicative $k$-form $\theta$ such that $d\theta= \omega$.
\end{enumerate}
Thus, in order to reconstruct $\theta$ it suffices to apply the previous two theorems.
\end{proof}

We now discuss the proof. Let $\sigma= V(\omega)$ and let
us first look at solutions $\xi\in W^{1, k-1}(A)$ of the equations:
\begin{equation}
\label{xi-eq} d^{\textrm{v}}(\xi)= \sigma, d^{\textrm{h}}(\xi)= 0.
\end{equation}
We use the same formulas as in the proof of Proposition \ref{int-pr}
(but applied to $\xi$ instead of $\sigma$) to write out explicitly
the equations. For $d^{\textrm{v}}(\xi)= \sigma$ we find
\[ -d_{DR}(\xi_0(\alpha))= \sigma_0(\alpha), d_{DR}(\xi_1(\alpha))+ \xi_0(\alpha)= \sigma_1(\alpha).\]
As in the proof of Proposition \ref{int-pr}, we only have to
remember the second one. In other words, $d^{\textrm{v}}(\xi)=
\sigma$ tells us that $\xi_0$ is determined by $\xi_1$:
\begin{equation}
\label{xi0} \xi_0(\alpha)= \sigma_1(\alpha)- d_{DR}(\xi_1(\alpha)).
\end{equation}
The condition $d^{\textrm{h}}(\xi)= 0$ gives three equations,
corresponding to the three components. For $(d^{\textrm{h}}(\xi))_2=
0$ we find that $\xi_1$ must satisfy the anti-symmetry condition
(\ref{mmk-1}). For $(d^{\textrm{h}}(\xi))_1= 0$ we find:
\[ \xi_1([\alpha, \beta])= L_{\rho(\alpha)}(\xi_1(\beta))+ i_{\rho(\beta)}(\xi_0(\alpha)).\]
Using the formula for $\xi_0$ in terms of $\xi_1$, we find that
$\xi_1$ must satisfy (\ref{mmk-2}).

Next, if $\theta$ is as in the theorem, we have $\delta \theta= 0$
and $\ud \theta= -\omega$, i.e. $\xi:= - V(\theta)$ must satisfy
(\ref{xi-eq}). The previous discussion shows that $l= \xi_1$ must
satisfy the equations above. From the definition of the Van Est map
it follows that $l(\alpha)= - J_{\alpha}(\theta)= -
i_{\alpha}(\theta)|_{M}$.

Assume now that $l$ satisfies the equations from the statement. Let
$\xi\in W^{1, k-1}(A)$ with $\xi_1= l$ and $\xi_0$ defined by
(\ref{xi0}), so that $\xi$ satisfies (\ref{xi-eq}). Using the Van
Est isomorphism, we find $\theta'\in \Omega^{k-1}(G)$ multiplicative
and $\eta\in \Omega^{k-1}(M)$ such that $V(\theta')= - \xi+
d^{\textrm{h}}(\eta)$. Choose $\theta= \theta'- \delta\eta$. Then
\[ V(\ud(\theta)- \omega)= - d^{\textrm{v}}(V(\theta))- V(\omega)= -d^{\textrm{v}}(-\xi+ d^{\textrm{h}}\eta- V(d^{\textrm{h}}\eta))- \sigma= d^{\textrm{v}}\xi- \sigma= 0.\]
On the other hand,  $\ud(\theta)- \omega$ is both multiplicative and
closed and therefore Lemma \ref{l-uniq} implies that $\omega=
\ud(\theta)$. By construction, the $l$ corresponding to $\theta$ is
the $l$ we started with, concluding the proof of the surjectivity.
For the injectivity, one proceeds exactly as in the proof of Theorem
\ref{cohomological integration1}. If $\theta$ and $\theta'$ have the
same associated $l$ and transgress $\omega$, $\theta- \theta'$ will
be multiplicative and closed with $V(\theta- \theta')_1= l-l= 0$. In
this case Lemma \ref{l-uniq} implies that $\theta= \theta'$.

%%%%%%%%%%%%%%%%%%%%%%%%%%%%%%%%%%%%%%%%%%%%%%%%%%%%%%%%
%%%%%%%%%%%%%%%%%%%%%%%%%%%%%%%%%%%%%%%%%%%%%%%%%%%%%%%%
%%%%%%%%%%%%%%%%%%%%%%%%%%%%%%%%%%%%%%%%%%%%%%%%%%%%%%%%
%%%%%%%%%%%%%%%%%%%%%%%%%%%%%%%%%%%%%%%%%%%%%%%%%%%%%%%%
\section{Appendix: Kalkman's BRST algebra in the infinite dimensional case}
\label{appendix}
%%%%%%%%%%%%%%%%%%%%%%%%%%%%%%%%%%%%%%%%%%%%%%%%%%%%%%%%
%%%%%%%%%%%%%%%%%%%%%%%%%%%%%%%%%%%%%%%%%%%%%%%%%%%%%%%%
%%%%%%%%%%%%%%%%%%%%%%%%%%%%%%%%%%%%%%%%%%%%%%%%%%%%%%%%
%%%%%%%%%%%%%%%%%%%%%%%%%%%%%%%%%%%%%%%%%%%%%%%%%%%%%%%%

In this paper we use some constructions which, although standard in
the finite dimensional case, need some clarification in the infinite
dimensional setting. Here we make these clarifications and we fix
our notations. In particular, we will give an intrinsic description
of Kalkman's BRST algebra which applies also to infinite dimensional
Lie algebras and more general coefficients.

\subsection{Chevalley-Eilenberg complexes} For a representation $V$ of a Lie
algebra $\mathfrak{g}$, the action $\mathfrak{g}\otimes V\rmap V$ is
denoted by $(\alpha, v)\mapsto L_{\alpha}(v)$. The
Chevalley-Eilenberg complex with coefficients in $V$ is
\[ \Lambda(\mathfrak{g}^*, V):= \textrm{Hom}_{\mathbb{R}}(\Lambda \mathfrak{g}, V) ,\]
where ``$\textrm{Hom}_{\mathbb{R}}$'' stands for the space of
$\mathbb{R}$-linear maps. In degree $k$, $\Lambda^k (\mathfrak{g}^*,
V)$ consists of antisymmetric multilinear maps depending on
$k$-variables from $\mathfrak{g}$, with values in $V$. The
Chevalley-Eilenberg differential,
\[ \delta: \Lambda^p (\mathfrak{g}^*, V)\rmap \Lambda^{p+1}(\mathfrak{g}^*, V)\]
is given by the Koszul formula:

\begin{eqnarray*}
(\delta(c))(\alpha_1, \dots ,\alpha_{p+1})&=& \sum_{i<j}(-1)^{i+j} c([\alpha_i,\alpha_j],\dots,
\hat{\alpha_i},\dots,\hat{\alpha_j},\dots,\alpha_{p+1})+\\
&+&\sum_i(-1)^{i+1}L_{\rho(\alpha_i)}(c(\alpha_1,\dots,\hat{\alpha_i},\dots,\alpha_{p+1})).
\end{eqnarray*}

\subsection{Symmetric powers} We now specify our conventions and notations
regarding symmetric powers. For any two vector spaces $E$ and $V$,
the space of $V$-valued polynomials on $E$ is
\[ S(E^*, V):= \textrm{Hom}_{\mathbb{R}}( SE, V).\]
where ``$\textrm{Hom}_{\mathbb{R}}$'' stands for the space of
$\mathbb{R}$-linear maps. A polynomial of degree $k$ will be viewed
either as a symmetric $k$-multilinear map
\[ P: \underbrace{E\times \ldots \times E}_{k\ \textrm{times}} \rmap V\]
or as an actual function on $E$ with values in $V$: $P(\alpha)=
P(\alpha, \ldots , \alpha)$. We will also use the following
operation. For each $\alpha\in E$ there is a partial derivative
\[ \partial_{\alpha}: S^k(E^*, V)\rmap S^{k-1}(E^*, V),\]
\[ \partial_{\alpha}(P)(\alpha_0):= \frac{d}{dt} |_{t= 0} P(\alpha_0+ t\alpha).\]
Here the derivative should be interpreted formally. In the
multilinear notation this operation is:
\[ \partial_{\alpha}(P)(\alpha_0)= k P(\alpha, \alpha_0, \ldots , \alpha_0) .\]

\subsection{Some representations} For any representation $V$ of
$\mathfrak{g}$, $S(\mathfrak{g}^*; V)$ is itself a representation in
a canonical way. The action
\[ \mathfrak{g}\otimes S(\mathfrak{g}^*; V)\rmap S(\mathfrak{g}^*; V),\ \ (\alpha, P)\mapsto L_{\alpha}(P)\]
is induced from the coadjoint action on $\mathfrak{g}^*$ and the
given action on $V$. These, together with the Leibniz identity for
$L_{\alpha}$, determine the action uniquely in the finite
dimensional case. We take the resulting explicit formula as the
definition in the general case:
\[ L_{\alpha}(P)(\alpha_1, \ldots , \alpha_p)= L_{\alpha}(P(\alpha_1, \ldots , \alpha_p))- \sum_{i} P(\alpha_1, \ldots ,
 [\alpha, \alpha_i], \ldots , \alpha_p).\]

A related representation arises in the case when $\mathfrak{g}=
\Gamma(A)$ is the Lie algebra of sections of a Lie algebroid $A$
over $M$. It is the algebra $\Omega(M; SA^*)$ of forms on $M$ with
values in the symmetric algebra of $A$. The action $(\alpha,
\omega)\mapsto L_{\alpha}(\omega)$
% \[ \mathfrak{g}_A\otimes \Omega(M; SA^*) \rmap \Omega(M; SA^*), \ \ (\alpha, \omega)\mapsto L_{\alpha}(\omega)\]
is uniquely determined by
\begin{itemize}
\item the Leibniz derivation identity: for all $\omega, \omega'\in \Omega(M; SA^*)$,
\[ L_{\alpha}(\omega \omega')= L_{\alpha}(\omega)\omega'+ \omega L_{\alpha}(\omega').\]
\item on $\Omega(M)$, $L_{\alpha}$
coincides with the usual Lie derivative $L_{\rho(\alpha)}$ along the
vector field $\rho(\alpha)$.
\item on $\Gamma(A^*)$, $L_{\alpha}$ is given by $L_{\alpha}(\xi)(\beta)= L_{\alpha}(\xi(\beta))- \xi(L_{\alpha}(\beta))$.
\end{itemize}
Actually, $\Omega(M; SA^*)$ is just a sub-representation of
$S(\mathfrak{g}^*; V)$ with $V= \Omega(M)$.

\subsection{The Weil algebra with coefficients} Assume now that
$\mathfrak{g}$ is a Lie algebra and $\mathcal{A}$ is a
$\mathfrak{g}$-DG algebra. We define
\[ W(\mathfrak{g}; \mathcal{A}):= \Lambda (\mathfrak{g}^*; S(\mathfrak{g}^*, \mathcal{A}))\]
with the following bi-grading:
\[ W^{p, q}(\mathfrak{g}; \mathcal{A}):= \bigoplus_{k} \Lambda^{p-k}(\mathfrak{g}^*;  S^k(\mathfrak{g}^*, \mathcal{A}^{q-k})).\]
For an element $c\in \Lambda^{p-k}(\mathfrak{g}^*;
S^k(\mathfrak{g}^*, \mathcal{A}^{q-k}))$ we use the notation
\[ c(\alpha_1, \ldots , \alpha_{p-k}|\alpha):= c(\alpha_1, \ldots , \alpha_{p-k})(\alpha)\in \mathcal{A}^{q-k},\]
which is an expression multilinear antisymmetric in its first
entries and polynomial in the last entry. $W(\mathfrak{g};
\mathcal{A})$ has a product structure compatible with the
bi-grading. For
\[ c\in \Lambda^p(\mathfrak{g}^*; S^k(\mathfrak{g}^*, \mathcal{A}^q)), c'\in \Lambda^{p'}
(\mathfrak{g}^*;
 S^{k'}(\mathfrak{g}^*, \mathcal{A}^{q'})),\] $cc'$
is given by
\[ (cc')(\alpha_1, \ldots , \alpha_{p+p'}|\alpha)= (-1)^{q p'}\sum \textrm{sgn}(\sigma) c(\alpha_{\sigma(1)}, \ldots, \alpha_{\sigma(p)}| \alpha) c'(\alpha_{\sigma(p+ 1)}, \ldots, \alpha_{\sigma(p+ p')}| \alpha) ,\]
where the sum is over all $(p, p')$-shuffles. The sign in front of
the sum comes from the finite dimensional case and the standard sign
conventions: in $W(\mathfrak{g})\otimes \mathcal{A}$,
\[ (w\otimes a) (w'\otimes a')= (-1)^{\textrm{deg}(a) \textrm{deg}(w')} (ww'\otimes aa').\]

\subsection{Kalkman's BRST differentials} As in the case of the standard
Weil algebra, there are two differentials. The first one,
$d^{\textrm{h}}$, increases $p$  and is given by
\[ d^{\textrm{h}}(c)= \delta (c)+ i_{\mathcal{A}}(c).\]
Here $\delta$ is the Koszul differential, while
\[ i_{\mathcal{A}}:\Lambda^{p}(\mathfrak{g}^*,  S^k(\mathfrak{g}^*, \mathcal{A}^{q}))\rmap \Lambda^{p}(\mathfrak{g}^*,
 S^{k+1}(\mathfrak{g}^*, \mathcal{A}^{q-1}))\]
 is given by
\[ i_{\mathcal{A}}(c)(\alpha_1, \ldots , \alpha_{p}| \alpha)= (-1)^{p+1}i_{\alpha}(c(\alpha_1, \ldots ,\alpha_{p}| \alpha)).\]
Both $\delta$ and $i_{\mathcal{A}}$ are derivations (and that
motivates the sign in $i_{\mathcal{A}}$).

The second differential, $d^{\textrm{v}}$, increases $q$
 and is given by
\[ d^{\textrm{v}}(c)= d_{\mathcal{A}} (c)+ i_{\mathfrak{g}}(c).\]
Here $d_{\mathcal{A}}$ is given by
\[ d_{\mathcal{A}}(c)(\alpha_1, \ldots , \alpha_{p}|\alpha)= (-1)^{p}d_{\mathcal{A}}(c(\alpha_1, \ldots , \alpha_{p}|\alpha)),\]
 while
\[ i_{\mathfrak{g}}:\Lambda^{p}(\mathfrak{g}^*,  S^k(\mathfrak{g}^*, \mathcal{A}^{q}))\rmap \Lambda^{p-1}(\mathfrak{g}^*,  S^{k+1}(\mathfrak{g}^*, \mathcal{A}^{q}))\]
is given by
\[ i_{\mathfrak{g}}(c)(\alpha_1, \ldots , \alpha_{p-1}|\alpha)= (-1)^{p+1}c(\alpha_1, \ldots , \alpha_{p-1}, \alpha|\alpha ).\]
Again, both $d_{\mathcal{A}}$ and $i_{\mathfrak{g}}$ are
derivations.

In the case that $\mathfrak{g}$ is finite dimensional,
$W(\mathfrak{g}; \mathcal{A})= W(\mathfrak{g})\otimes \mathcal{A}$
and we can use a basis $e^a$ of $\mathfrak{g}$ to write the formulas
more explicitly.  We denote by $\theta^a$ the induced basis of
$\Lambda^1\mathfrak{g}^*$, by $\mu^a$ the induced basis of
$S^1\mathfrak{g}^*$ and by $d_{W}$ the differential of
$W(\mathfrak{g})$. From the derivation property of all the operators
$\delta, i_{\mathcal{A}}$, $d_{\mathcal{A}}$, $i_{\mathfrak{g}}$ and
after a straightforward checking on generators, we deduce that
\[ \delta= d^{\textrm{h}}_{W}\otimes 1  +\theta^a\otimes L_{e^a},\ \ i_{\mathcal{A}}= - \mu^a\otimes i_{e^a},\]
\[ d_{\mathcal{A}}= 1\otimes d_{\mathcal{A}}, \ \ i_{\mathfrak{g}}= d^{\textrm{v}}_{W}\otimes 1.\]
Hence, in this case, $d^{\textrm{h}}+ d^{\textrm{v}}$ does coincide
with Kalkman's differential on $W(\mathfrak{g})\otimes \mathcal{A}$.


\begin{thebibliography}{10}

\bibitem{Ari-Cra1} C.~Arias~Abad and M.~Crainic, Representations up to homotopy of {L}ie algebroids,
preprint arXiv:0901.0319, submitted for publication.


\bibitem{Ari-Cra2} C.~Arias~Abad and M.~Crainic, Representations up to homotopy and {B}ott's spectral sequence for Lie groupoids,
preprint arXiv:0911.2859, submitted for publication.





\bibitem{Bott} R.~Bott, On the {C}hern-{W}eil homomorphism and the continuous cohomology of {L}ie-groups,
\emph{Advances in Math.~}\textbf{11} (1973),  289--303.

\bibitem{Bott2} R.~Bott, H.~Shulman and J.~Stasheff, On the de Rham theory of certain classifying spaces
\emph{Advances in Math.~}\textbf{20} (1976), 43--56.


\bibitem{BUR} H.~Bursztyn, M.~Crainic, A.~Weinstein and C.~Zhu, Integration of twisted {D}irac brackets,
\emph{Duke Math. J.~}\textbf{123} (2004),  549--607.


\bibitem{Cartan} H.~Cartan, Notions d'algebre diff\'erentielle; applications aux groupes de Lie et aux ...,
In ''Colloque de Topologie'', C.B.R.M. Bruxelles, (1950), 15--27.



\bibitem{Cra1} M.~Crainic, Differentiable and algebroid cohomology, van Est isomorphisms, and characteristic classes,
\emph{Comment.~Math.~Helv.~}\textbf{78} (2003), no. 4, 681--721.


\bibitem{Cra2} M.~Crainic, Prequantization and {L}ie brackets,
\emph{ J.~Symplectic Geom.~}\textbf{2} (2004), 579--602.

\bibitem{Cra-Fer} M.~Crainic and R.~L.~Fernandes, Integrability of {L}ie brackets,
\emph{Ann. of Math. (2)} \textbf{157} (2003), 575--620.


\bibitem{Est} W.~T.~ van Est, Group cohomology and Lie algebra cohomology in Lie groups I, II,
\emph{Proc.~Kon.~Ned.~Aka.~}, \textbf{56} (1953), 484--504.

\bibitem{Getzler} N.~Berline, E.~ Getzler and M.~ Vergne, \emph{Heat kernels and Dirac operators,}
Grundlehren Text Editions. Springer-Verlag, Berlin, 2004

\bibitem{Getzler2}
Ezra Getzler.
\newblock The equivariant {C}hern character for non-compact {L}ie groups.
\newblock {\em Adv. Math.}, 109(1):88--107, 1994.


\bibitem{GS} V.~Guillemin and S.~Sternberg, \emph{Supersymmetry and equivariant de Rham theory,}
Springer-Verlag, Berlin, 1999.

\bibitem{Haef} A.~Haefliger, Groupo\"\i des d'holonomie et classifiants,
in Transversal structure of foliations (Toulouse, 1982),
\emph{Ast\'erisque}, \textbf{116} (1984), 70--97.

\bibitem{Kal} J.~Kalkman, B{RST} model for equivariant cohomology and representatives for the equivariant {T}hom class,
\emph{Comm. Math. Phys.}, \textbf{153} (1993), 447--463.

\bibitem{KT} F.~Kamber and P.~Tondeur, \emph{Foliated Bundles and Characteristic Classes, }
Springer-Verlag, Berlin-New York, 1975.
\bibitem{McK} K.~Mackenzie, \emph{General theory of Lie groupoids and Lie algebroids,}
London Mathematical Society Lecture Note Series 213, Cambridge University Press, Cambridge, 2005.


\bibitem{MQ} V.~Mathai and D.~Quillen, Superconnections, Thom classes, and equivariant differential forms, \emph{Topology.} \textbf{25} (1986), 85--110.




\bibitem{Meh} R.~Mehta, Supergroupoids, double structures and equivariant cohomology,
PhD Thesis, Berkeley (2006), Arxiv math/0605356.

\bibitem{Moe} I.~Moerdijk and J.~Mr{\v{c}}un, \emph{Introduction to foliations and {L}ie groupoids,} volume~91 of Cambridge Studies in Advanced Mathematics,
Cambridge University Press, Cambridge, 2003.

\bibitem{Sea} G.~Segal, Classifying spaces and spectral sequences,
\emph{Inst. Hautes \'Etudes Sci. Publ. Math.}, \textbf{34} (1968), 105--112.

\bibitem{Wein2} A.~Weinstein, Symplectic groupoids and Poisson manifolds,
\emph{Bull. Amer. Math. Soc.}, \textbf{16}  (1987), 101--104.

\bibitem{WeXu} A.~Weinstein and P.~Xu, Extensions of symplectic groupoids and quantization.
\emph{J. Reine Angew. Math.}, \textbf{417} (1991), 159--189.

\end{thebibliography}
\end{document}